\documentclass[12pt]{amsart}
\usepackage[top=1.25in, bottom=1.25in, left=1in, right=1in]{geometry}
\usepackage{amsmath}
\usepackage{enumerate}
\usepackage{float}
\usepackage{comment}
\usepackage{fancyvrb}
\usepackage{caption}
\usepackage{tikz-cd}
\usepackage{subcaption}

\usepackage{mathtools}
\usepackage{scalerel}
\usepackage{amssymb}
\usepackage{hyperref}
\hypersetup{
    colorlinks=true,
    linkcolor=blue,
    filecolor=magenta,      
    urlcolor=blue,
    citecolor=blue,
}
\usepackage{color}
\usepackage{graphicx}
\usepackage{mathrsfs}

\usepackage{mathabx}
\usepackage{amsthm}
\usepackage{color}

\usepackage[font=small]{caption}
\usetikzlibrary{cd}
\usetikzlibrary{decorations.markings}
\theoremstyle{definition}
\newtheorem{thm}{Theorem}[section]

\newtheorem{lemma}[thm]{Lemma}
\newtheorem{defn}[thm]{Definition}
\newtheorem{cor}[thm]{Corollary}
\newtheorem{rmk}[thm]{Remark}
\newtheorem{ex}[thm]{Example}

\def\droop{\operatorname{min-droop}}

\def\swap{\operatorname{cross-bump-swap}}

\def\BPD{\operatorname{BPD}}

\def\rect{\operatorname{rect}}

\def\jdt{\operatorname{jdt}}

\def\words{\operatorname{words}}

\def\+{\includegraphics[scale=0.4]{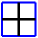}}
\def\bl{\includegraphics[scale=0.4]{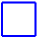}}
\def\bt{\includegraphics[scale=0.4]{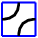}}
\def\rt{\includegraphics[scale=0.4]{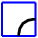}}
\def\jt{\includegraphics[scale=0.4]{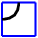}}

\def\mindroop{\textsf{min-droop}}
\def\maxdroop{\textsf{max-droop}}
\def\maxd{\operatorname{max-droop}}

\def\cbswap{\textsf{cross-bump-swap}}

\def\recdroop{\textsf{rec-droop}}
\def\recd{\operatorname{rec-droop}}
\def\recundroop{\textsf{rec-undroop}}
\def\recund{\operatorname{rec-undroop}}
\def\uncross{\textsf{uncross}}
\def\cross{\textsf{cross}}
\def\crss{\operatorname{cross}}
\def\rect{\operatorname{rect}}
\def\pop{\operatorname{pop}}

\title{Growth Diagrams for Schubert RSK}

\date{}

\author{Daoji Huang and Son Nguyen}
\thanks{DH was supported by NSF DMS-2202900. SN was partially supported by the University of Minnesota's Office of Undergraduate Research.}
\begin{document}
\begin{abstract}
    Motivated by classical combinatorial Schubert calculus on the Grassmannian, Huang--Pylyavskyy introduced a generalized theory of Robinson-Schensted-Knuth (RSK) correspondence for studying Schubert calculus on the complete flag variety via insertion algorithms. The inputs of the correspondence are certain biwords, the insertion objects are bumpless pipe dreams, and the recording objects are certain chains in Bruhat order. In particular, they defined plactic biwords and showed that classical Knuth relations can be generalized to plactic biwords. In this paper, we give an analogue of Fomin's growth diagrams for this generalized RSK correspondence on plactic biwords. We show that this growth diagram recovers the bijection between pipe dreams and bumpless pipe dreams of Gao--Huang.
\end{abstract}
\maketitle
\section{Introduction}

The general philosophy of a \emph{growth diagram} can be thought of as translating a temporal object, i.e., an algorithm, to a spatial object, i.e., a diagrammatic encoding of the algorithm, so as to provide a powerful tool to study the algorithm, as well as an interface between combinatorial algorithms and algebraic or geometric phenomena.\footnote{We learned this philosophy from Allen Knutson.} The most classical example of a growth diagram is of the classical Robinson-Schensted (RS) correspondence, a bijection between a permutation and a pair of standard Young tableaux. 
The Robinson-Schensted-Knuth (RSK) correspondence is a generalization of the RS correspondence and is of central importance in symmetric function theory. Each variation of these correspondences has its corresponding growth diagram version.  
The RS correspondence is originally defined as an insertion algorithm on pairs of standard tableaux. The algorithm iteratively scans the permutation, inserting each time a number to the insertion tableaux, and records the position of the new entry in the recording tableaux. The growth diagram first introduced by Fomin \cite{fomin1994duality,fomin1995schensted}, however, is a two dimensional grid that can be roughly thought of as an ``enriched'' permutation matrix, with the extra information determined by certain local ``growth rules.'' Although far from apparent at a first glance, the growth diagram is a lossless encoding of the insertion algorithm. Furthermore, the growth diagram manifests many non-obvious properties of the insertion algorithm. For example, the property $w\xleftrightarrow{RS}(P,Q)$ implies $w^{-1}\xleftrightarrow{RS}(Q,P)$ can be easily seen by transposing the growth diagram. 

It is possible to give the RSK correspondence an operator theoretic interpretation through growth diagrams and, as a consequence, obtain a noncommutative version of Cauchy's identity \cite{fomin1995schur}. Furthermore, growth diagrams for the RS correspondence has beautiful geometric and representation-theoretic interpretations \cite{van2000flag,rosso2010robinson,steinberg1988occurrence}. 

Beyond classical RSK, there are many examples in the literature of expressing combinatorial algorithms using growth diagrams, see, e.g., \cite{lam2010affine,lenart2010growth,patrias2018dual,thomas2009jeu}.

In \cite{HP} and \cite{huang2023knuth}, the first author and Pylyavskyy introduced a generalization of the classical RSK correspondence for Schubert polynomials, called bumpless pipe dream (BPD) RSK. For the definition  of bumpless pipe dreams, we refer the readers to \cite{LLS}.  As in the classical case, this generalization of RSK is defined via insertion algorithms. The algorithm takes as input a certain biword, iteratively inserts it into a bumpless pipe dream, and records the insertion via chains in mixed $k$-Bruhat order. An analogue of Knuth moves was discovered for a more restrictive set of biwords, called \emph{plactic biwords}. It is then natural to pursue a growth diagram version of this generalized RSK correspondence on plactic biwords. In this paper, we describe these new growth diagrams for the RSK correspondence for plactic biwords. As an application, our growth diagram manifests the canonical bijection between pipe dreams and bumpless pipe dreams of the first author and Gao \cite{gao2023canonical}. We also hope that this opens up a venue for connecting the combinatorics of this generalized RSK to its algebraic or geometric interpretations.
\vskip 0.5em
\noindent\textbf{Acknowledgements.} We thank Pavlo Pylyavskyy for helpful conversations.

\section{Plactic biwords and growth rules}
\subsection{Generalized Knuth relations on plactic biwords}
In this paper, we inherit the convention of \cite{HP} and use $\BPD(\pi)$ to denote the set of BPDs for permutation $\pi$.
\begin{defn}[\cite{huang2023knuth}]
\label{def:assoc-biwords}
A biletter is a pair of positive integers $\binom{a}{k}$ where $a\le k$.
A \textbf{plactic biword} is a word of biletters $\binom{\mathbf{a}}{\mathbf{k}}=\binom{a_1,\cdots, a_\ell}{k_1,\cdots, k_\ell}$, where $k_i\ge k_{i+1}$ for each $i$. 
\end{defn}

\begin{defn}[\cite{huang2023knuth}]
\label{def:knuth}
We define the \textbf{generalized Knuth relations} on plactic biwords as follows:
\begin{enumerate}
    \item[(1)] $\left(\begin{smallmatrix}\cdots & b & a & c & \cdots \\ \cdots & k & k & k & \cdots\end{smallmatrix}\right)\sim \left(\begin{smallmatrix}\cdots & b & c & a & \cdots \\ \cdots & k & k & k & \cdots\end{smallmatrix}\right)$ if $a<b\le c$
    \item[(2)]$\left(\begin{smallmatrix}\cdots & a & c & b & \cdots \\ \cdots & k & k & k & \cdots\end{smallmatrix}\right)\sim \left(\begin{smallmatrix}\cdots & c & a & b & \cdots \\ \cdots & k & k & k & \cdots\end{smallmatrix}\right)$ if $a\le b< c$
    \item[(3)] $\left(\begin{smallmatrix}\cdots & a & b & \cdots \\ \cdots & k & k & \cdots\end{smallmatrix}\right)\sim \left(\begin{smallmatrix}\cdots & a & b & \cdots \\ \cdots & k+1 & k & \cdots\end{smallmatrix}\right)$ if $a\le b$
    \item[(4)] $\left(\begin{smallmatrix}\cdots & b & a & \cdots \\ \cdots & k+1 & k+1 & \cdots\end{smallmatrix}\right)\sim \left(\begin{smallmatrix}\cdots & b & a & \cdots \\ \cdots & k+1 & k & \cdots\end{smallmatrix}\right)$ if $a< b$.
\end{enumerate}

\end{defn}
Notice that these relations are only defined on plactic biwords. We do not apply the relation (3) or (4) if the resulting word is no longer plactic.

Recall that a given biword $Q=\binom{a_1,\cdots,a_\ell}{k_1,\cdots,k_\ell}$, \cite{HP} defines a map $\mathcal{L}(Q)=(\varphi_L(Q), ch_L(Q))$ where $\varphi_L(Q)$ is BPD obtained by reading $Q$ from right to left and successively performing left insertion, and %$ch_L(Q)=(\id\lessdot_{k_\ell} \pi_1\lessdot_{k_\ell-1}\cdots \lessdot_{k_1}\pi_\ell=\perm(\varphi_L(D)))$,
$ch_L(Q)$ is the recording chain in mixed $k$-Bruhat order with edge labels $k_\ell,\cdots, k_1$,
as well as a map $\mathcal{R}(Q)=(\varphi_R(Q), ch_R(Q))$ where $\varphi_R(Q)$ is the BPD obtained by reading $Q$ from left to right and successively performing right insertion, and $ch_R(Q)$ is the recording chain in mixed $k$-Bruhat order with edge labels $k_1,\cdots, k_\ell$.
For details of these insertion algorithms see \cite[Section 3]{HP}. We recall the right insertion algorithm in Section~\ref{sec:insert-jdt}. Furthermore, by \cite[Proposition 1.2]{huang2023knuth}, the insertion BPD is well-defined regardless of the choice of insertion algorithms, so we write $\varphi(D):=\varphi_R(D)=\varphi_L(D)$. For the analysis of the insertion algorithm in this paper we use $\mathcal{R}$, the right insertion algorithm.
\begin{thm}[\cite{huang2023knuth}]
\label{thm:main}
For any $D\in\BPD(\pi)$, the set of plactic biwords
\[\words(D):=\{Q: \varphi(Q)=D\}\] is connected by the generalized Knuth relations.
\end{thm}

For a biword $Q$, we define $Q_{> i}$ to be the biword obtained from $Q$ by removing all biletters $\binom{a_j}{k_j}$ with $a_j \leq i$. In particular, $Q_{>0}$ is $Q$. We have the following lemma.

\begin{lemma}\label{lem:remove_i}
    Suppose $Q$ and $Q'$ are connected by the generalized Knuth relations, then for all $i$, $Q_{>i}$ and $Q'_{>i}$ are connected by the generalized Knuth relations.
\end{lemma}

\begin{proof}
    It suffices to consider the case where $Q$ and $Q'$ are connected by one generalized Knuth relation. Observe that in all relations, if we remove the biletters $\binom{a}{k}$ and $\binom{a}{k+1}$, then the remaining biwords are the same. Thus, we can iteratively remove all biletters $\binom{1}{*}, \binom{2}{*},\ldots,\binom{i}{*}$, and after each step, either the remaining biwords are connected by the same generalized Knuth relation or they are the same biword.
 
\end{proof}

As a result of Lemma \ref{lem:remove_i}, for any $D\in \BPD(\pi)$ and any $i$, the set of plactic words $\{ Q_{> i} ~|~ Q \in \words(D) \}$ is also connected by the generalized Knuth relations. Therefore, for any $Q\in\words(D)$, $\varphi(Q_{>i})$ is the same BPD. 

\begin{rmk}
    One could similarly define $Q_{<i}$ to be the biword obtained from $Q$ by removing all biletters $\binom{a_j}{k_j}$ with $a_j \geq i$ and ask if $Q \sim Q'$ implies $Q_{<i} \sim Q'_{<i}$ for all $i$. The answer is unfortunately no. One small example is $\left(\begin{smallmatrix} 1 & 3 & 2 \\ 3 & 3 & 3 \end{smallmatrix}\right)\sim \left(\begin{smallmatrix} 1 & 3 & 2 \\ 3 & 3 & 2 \end{smallmatrix}\right)$ but $\left(\begin{smallmatrix} 1 & 2 \\ 3 & 3 \end{smallmatrix}\right)$ and $\left(\begin{smallmatrix} 1 & 2 \\ 3 & 2 \end{smallmatrix}\right)$ are not connected by generalized Knuth relations. The reason is that if $Q$ and $Q'$ are connected by the generalized Knuth relation (3) or (4), then removing $\binom{b}{*}$ yields two different, non-equivalent biwords.
\end{rmk}

\subsection{Jeu de taquin on BPDs}
 Given $D\in\BPD(\pi)$ with $\ell(\pi)>0$, \cite[Definition 3.1]{gao2023canonical} produces another bumpless pipe dream $\nabla D\in\BPD(\pi')$ where $\ell(\pi')=\ell(\pi)-1$. We recall this definition in Section~\ref{sec:insert-jdt}. We call the $\nabla$ operator \textbf{jeu de taquin} on BPDs. 
The justification of this name is that, after applying a direct bijection between (skew) semistandard tableaux and BPDs for Grassmaninan permutations, the jeu de taquin algorithm on tableaux can be realized as a corresponding algorithm on BPDs. See \cite{huang2021schubert} for a detailed description. We will sometimes use the notation $\jdt(b,r)$ instead of $\nabla$ to emphasize that jeu de taquin starts from  position $(b,r)$. See Figure~\ref{fig:growth-ex3} for an illustration.

     For each BPD $D$, let $b$ be the smallest row with an empty square $\bl$, define $D' = \rect(D)$ to be the BPD obtained from $D$ by performing $\jdt$ on all empty squares on row $b$ from right to left. Suppose $\pi$ and $\mu$ are the permutations of $D'$ and $D$, respectively, then by \cite{gao2023canonical}, we have
    \[ \mu = s_{i_j}\ldots s_{i_1}\pi \]
    where $i_j>\ldots>i_1$. Thus, we define $I(D) = \{i_1,\ldots,i_j\}$.
    Also, when there is little ambiguity, we denote the BPD corresponding to a permutation $\pi$ on the growth diagram as $D_\pi$.

    \begin{thm}\label{thm:jdt}
        Let $D$ be the BPD corresponding to a biword 
        $w=\left( \begin{smallmatrix}
            b_1 & b_2 & \ldots & b_\ell \\
            k_1 & k_2 & \ldots & k_\ell \\
        \end{smallmatrix} \right)$
        and $b = \min\{b_1,\ldots,b_\ell\}$, and let $D'$ be the BPD corresponding to $w'$ obtained by removing all biletter $\binom{b}{k_i}$ from $w$. Then $D' = \rect(D)$.
    \end{thm}

    The following corollary is immediate from Theorem \ref{thm:jdt} by \cite{gao2023canonical}.

    \begin{cor}\label{cor:jdt-si}
        With the same notation as in Theorem \ref{thm:jdt}, let $\pi$ and $\mu$ be the permutations of $D'$ and $D$, respectively, then
        \[ \mu = s_{i_j}\ldots s_{i_1}\pi \]
        where $i_j>\ldots>i_1$.
    \end{cor}
\subsection{Growth diagrams}
\subsubsection{Defining growth diagrams}
Given a plactic biword
    $\left( \begin{smallmatrix}
        b_1 & b_2 & \ldots & b_\ell \\
        k_1 & k_2 & \ldots & k_\ell \\
    \end{smallmatrix} \right)$
    and let $a = \max\{b_i~|~ 1\leq i\leq \ell\}$. We define a growth diagram to be a matrix of permutations $\pi_{i,j}$ with $0 \leq i \leq a$ and $0 \leq j \leq \ell$. The \textbf{initial condition} is $\pi_{i,0} = \text{id}$ for all $i$ and $\pi_{a,j} = \text{id}$ for all $j$. The figure below shows a generic \textbf{square} of the growth diagram.

    \[
    \begin{tikzcd}[sep=tiny]
    	\pi_{i,j-1} && \pi_{i,j} \\
    	& {\ } \\
    	\pi_{i-1,j-1} && \pi_{i-1,j}
    	\arrow[no head, from=1-1, to=1-3]
    	\arrow[no head, from=1-1, to=3-1]
    	\arrow[no head, from=3-1, to=3-3]
    	\arrow[no head, from=1-3, to=3-3]
    \end{tikzcd}
    \]
    We fill the squares of the growth diagram as follows. For each biletter $\binom{b_i}{k_i}$, we put an $\times_{k_i}$ in the square whose corners are $\pi_{b_i,i-1},\pi_{b_i,i},\pi_{b_i-1,i-1},\pi_{b_i-1,i}$. In addition, in every other square between columns $i-1$ and $i$, we put a subscript $k_i$. The following figure shows an example where the biword is
    $\left( \begin{smallmatrix}
        1 & 3 & 1 & 2 & 1 \\
        3 & 3 & 2 & 2 & 1 \\
    \end{smallmatrix} \right).$
    \[\begin{tikzcd}[sep=tiny]
	{\pi_{3,0}} && {\pi_{3,1}} && {\pi_{3,2}} && {\pi_{3,3}} && {\pi_{3,4}} && {\pi_{3,5}} \\
	& _3 && \textcolor{red}{\times_3} && _2 && _2 && _1 \\
	{\pi_{2,0}} && {\pi_{2,1}} && {\pi_{2,2}} && {\pi_{2,3}} && {\pi_{2,4}} && {\pi_{2,5}} \\
	& _3 && _3 && _2 && \textcolor{red}{\textcolor{red}{\times_2}} && _1 \\
	{\pi_{1,0}} && {\pi_{1,1}} && {\pi_{1,2}} && {\pi_{1,3}} && {\pi_{1,4}} && {\pi_{1,5}} \\
	& \textcolor{red}{\times_3} && _3 && \textcolor{red}{\times_2} && _2 && \textcolor{red}{\times_1} \\
	{\pi_{0,0}} && {\pi_{0,1}} && {\pi_{0,2}} && {\pi_{0,3}} && {\pi_{0,4}} && {\pi_{0,5}}
	\arrow[no head, from=7-3, to=7-5]
	\arrow[no head, from=7-5, to=7-7]
	\arrow[no head, from=7-7, to=7-9]
	\arrow[no head, from=7-1, to=7-3]
	\arrow[no head, from=7-1, to=5-1]
	\arrow[no head, from=5-1, to=3-1]
	\arrow[no head, from=3-1, to=1-1]
	\arrow[no head, from=1-1, to=1-3]
	\arrow[no head, from=1-3, to=1-5]
	\arrow[no head, from=1-5, to=1-7]
	\arrow[no head, from=1-7, to=1-9]
	\arrow[no head, from=1-9, to=1-11]
	\arrow[no head, from=1-3, to=3-3]
	\arrow[no head, from=3-3, to=5-3]
	\arrow[no head, from=5-3, to=7-3]
	\arrow[no head, from=1-5, to=3-5]
	\arrow[no head, from=3-5, to=5-5]
	\arrow[no head, from=5-5, to=7-5]
	\arrow[no head, from=5-1, to=5-3]
	\arrow[no head, from=5-3, to=5-5]
	\arrow[no head, from=5-5, to=5-7]
	\arrow[no head, from=5-7, to=5-9]
	\arrow[no head, from=5-9, to=5-11]
	\arrow[no head, from=1-11, to=3-11]
	\arrow[no head, from=3-11, to=5-11]
	\arrow[no head, from=5-11, to=7-11]
	\arrow[no head, from=7-9, to=5-9]
	\arrow[no head, from=3-11, to=3-9]
	\arrow[no head, from=3-9, to=3-7]
	\arrow[no head, from=3-7, to=3-5]
	\arrow[no head, from=3-5, to=3-3]
	\arrow[no head, from=3-3, to=3-1]
	\arrow[no head, from=1-7, to=3-7]
	\arrow[no head, from=3-7, to=5-7]
	\arrow[no head, from=5-7, to=7-7]
	\arrow[no head, from=7-9, to=7-11]
	\arrow[no head, from=5-9, to=3-9]
	\arrow[no head, from=3-9, to=1-9]
    \end{tikzcd}\]

    For each point $(i,j)$ in the growth diagram, let $w(i,j)$ be the biword obtained from reading from left to right the X's to the NW of $(i,j)$. Formally speaking, $w(i,j)$ is obtained from 
    $\left( \begin{smallmatrix}
        b_1 & b_2 & \ldots & b_\ell \\
        k_1 & k_2 & \ldots & k_\ell \\
    \end{smallmatrix} \right)$
    by removing all biletter $\binom{b_s}{k_s}$ with $b_s \leq i$ or $s> j$. For example, in the above growth diagram, $w(1,4) = \left( \begin{smallmatrix}
        3 & 2 \\
        3 & 2 \\
    \end{smallmatrix} \right)$. Define $\pi_{i,j}$ to be the permutation of $\varphi(w(i,j))$, the bumpless pipe dream obtained by inserting $w(i,j)$.

\begin{rmk}
    When $k_1=\cdots =k_\ell=k$, we recover a version of classical growth diagrams for the RSK correspondence, where the input is a word with letters in positive numbers, the insertion object is a semistandard tableau, and the recording object is a standard tableau. For classical Knuth relations, deleting either all of the smallest letter in a word, or all of the largest letter in a word, preserves Knuth classes. 
\end{rmk}

\subsubsection{Local rules}
    \begin{thm}\label{thm:local-rule}
        Given a square with subscript $k$ as follows:

        \[
        \begin{tikzcd}[sep=tiny]
        	\pi && \sigma \\
        	& {\ } \\
        	\mu && \rho
        	\arrow[no head, from=1-1, to=1-3]
        	\arrow[no head, from=1-1, to=3-1]
        	\arrow[no head, from=3-1, to=3-3]
        	\arrow[no head, from=1-3, to=3-3]
        \end{tikzcd}
        \]
        Then one can get $\rho$ from $\pi,\mu$, and $\sigma$ by the following rules:

        \begin{enumerate}
            \item If there is no $\times$:
            \begin{enumerate}
                \item[(a)] If $\pi = \sigma$ then $\rho = \mu$.
                \item[(b)] If $\pi = \mu$ then $\rho = \sigma$.
                \item[(c)] If $\pi \neq \sigma, \mu$, then  $\mu = s_{i_j}\ldots s_{i_1}\pi$ where $I = \{i_j>\ldots>i_1\}$, and $\sigma = t_{\alpha\beta}\pi $ such that $\pi^{-1}(\alpha)\le k <\pi^{-1}(\beta)$ for some $\alpha<\beta$. 
                 Let $x := \min(I^C \cap [\alpha,\beta))$, and $A := (I^C\cap [\beta,\infty)) \cup \{x\} = \{j_1<j_2<\ldots\}$. Then $\rho = t_{j_\ell,j_{\ell+1}}\mu$ where $\ell$ is the smallest index such that $\mu^{-1}(j_\ell) \leq k < \mu^{-1}(j_{\ell+1})$.
            \end{enumerate}
            \item If there is an $\times$, then $\pi = \sigma$ and $\mu = s_{i_j}\ldots s_{i_1}\pi$ where $I = \{i_j>\ldots>i_1\}$. Let $I^C = \{j_1<j_2<\ldots\}$, then $\rho = t_{j_\ell,j_{\ell+1}}\mu$ where $\ell$ is the smallest index such that $\mu^{-1}(j_\ell) \leq k < \mu^{-1}(j_{\ell+1})$.
        \end{enumerate}
    \end{thm}

    \begin{ex}
    \label{ex:growth}
     Let the biword be
    $\left( \begin{smallmatrix}
        1 & 3 & 1 & 2 & 1 \\
        3 & 3 & 2 & 2 & 1 \\
    \end{smallmatrix} \right),$
    using the rules in Theorem \ref{thm:local-rule}, we have the following growth diagram.
    \[\begin{tikzcd}[sep=tiny]
	12345 && 12345 && 12345 && 12345 && 12345 && 12345 \\
	& _3 && \textcolor{red}{\times_3} && _2 && _2 && _1 \\
	12345 && 12345 && 12435 && 12435 && 12435 && 12435 \\
	& _3 && _3 && _2 && \textcolor{red}{\times_2} && _1 \\
	12345 && 12345 && 12435 && 12435 && 13425 && 13425 \\
	& \textcolor{red}{\times_3} && _3 && \textcolor{red}{\times_2} && _2 && \textcolor{red}{\times_1} \\
	12345 && 12435 && 12534 && 13524 && 15324 && 25314
	\arrow[no head, from=7-3, to=7-5]
	\arrow[no head, from=7-5, to=7-7]
	\arrow[no head, from=7-7, to=7-9]
	\arrow[no head, from=7-1, to=7-3]
	\arrow[no head, from=7-1, to=5-1]
	\arrow[no head, from=5-1, to=3-1]
	\arrow[no head, from=3-1, to=1-1]
	\arrow[no head, from=1-1, to=1-3]
	\arrow[no head, from=1-3, to=1-5]
	\arrow[no head, from=1-5, to=1-7]
	\arrow[no head, from=1-7, to=1-9]
	\arrow[no head, from=1-9, to=1-11]
	\arrow[no head, from=1-3, to=3-3]
	\arrow[no head, from=3-3, to=5-3]
	\arrow[no head, from=5-3, to=7-3]
	\arrow[no head, from=1-5, to=3-5]
	\arrow[no head, from=3-5, to=5-5]
	\arrow[no head, from=5-5, to=7-5]
	\arrow[no head, from=5-1, to=5-3]
	\arrow[no head, from=5-3, to=5-5]
	\arrow[no head, from=5-5, to=5-7]
	\arrow[no head, from=5-7, to=5-9]
	\arrow[no head, from=5-9, to=5-11]
	\arrow[no head, from=1-11, to=3-11]
	\arrow[no head, from=3-11, to=5-11]
	\arrow[no head, from=5-11, to=7-11]
	\arrow[no head, from=7-9, to=5-9]
	\arrow[no head, from=3-11, to=3-9]
	\arrow[no head, from=3-9, to=3-7]
	\arrow[no head, from=3-7, to=3-5]
	\arrow[no head, from=3-5, to=3-3]
	\arrow[no head, from=3-3, to=3-1]
	\arrow[no head, from=1-7, to=3-7]
	\arrow[no head, from=3-7, to=5-7]
	\arrow[no head, from=5-7, to=7-7]
	\arrow[no head, from=7-9, to=7-11]
	\arrow[no head, from=5-9, to=3-9]
	\arrow[no head, from=3-9, to=1-9]
    \end{tikzcd}\]
    Notice that in the square
    \[
    \begin{tikzcd}[sep=tiny]
        \pi = 12435 && \sigma = 13425 \\
        & _2 \\
        \mu = 13524 && \rho = 15324
        \arrow[no head, from=1-1, to=1-3]
        \arrow[no head, from=1-1, to=3-1]
        \arrow[no head, from=3-1, to=3-3]
        \arrow[no head, from=1-3, to=3-3]
    \end{tikzcd}
    \]
    we use rule (1c) of Theorem \ref{thm:local-rule}. In particular, we have $\pi \neq \sigma, \mu$ and $\mu = s_4s_2\pi$. Thus, $I = \{2,4\}$. Also, $\sigma = t_{23}\pi$, so $A = \{3,5,6,\ldots\}$. Since $\mu^{-1}(3) \leq k = 2 < \mu^{-1}(5)$, we have $\rho = t_{35}\mu = 15324$. On the other hand, in the square
    \[
    \begin{tikzcd}[sep=tiny]
        \pi = 13425 && \sigma = 13425 \\
        & \times_1 \\
        \mu = 15324 && \rho = 25314
        \arrow[no head, from=1-1, to=1-3]
        \arrow[no head, from=1-1, to=3-1]
        \arrow[no head, from=3-1, to=3-3]
        \arrow[no head, from=1-3, to=3-3]
    \end{tikzcd}
    \]
    we use rule (2) of Theorem \ref{thm:local-rule}. We have $\mu = s_4s_3\pi$, so $I = \{3,4\}$. Thus, $I^C = \{1,2,5,6,\ldots\}$. We have $\mu^{-1}(1) \leq k = 1 < \mu^{-1}(2)$, so $\rho = t_{12}\mu = 25314$.
    \end{ex}

       To check that the above growth diagram is correct, we can go through the insertion process. Figure \ref{fig:growth-ex} shows the insertion process of this biword. One can check that the permutations we obtain along the way are exactly the permutations on the bottom row of the growth diagram.

    \begin{figure}[h!]
        \centering
        \includegraphics[scale = 0.4]{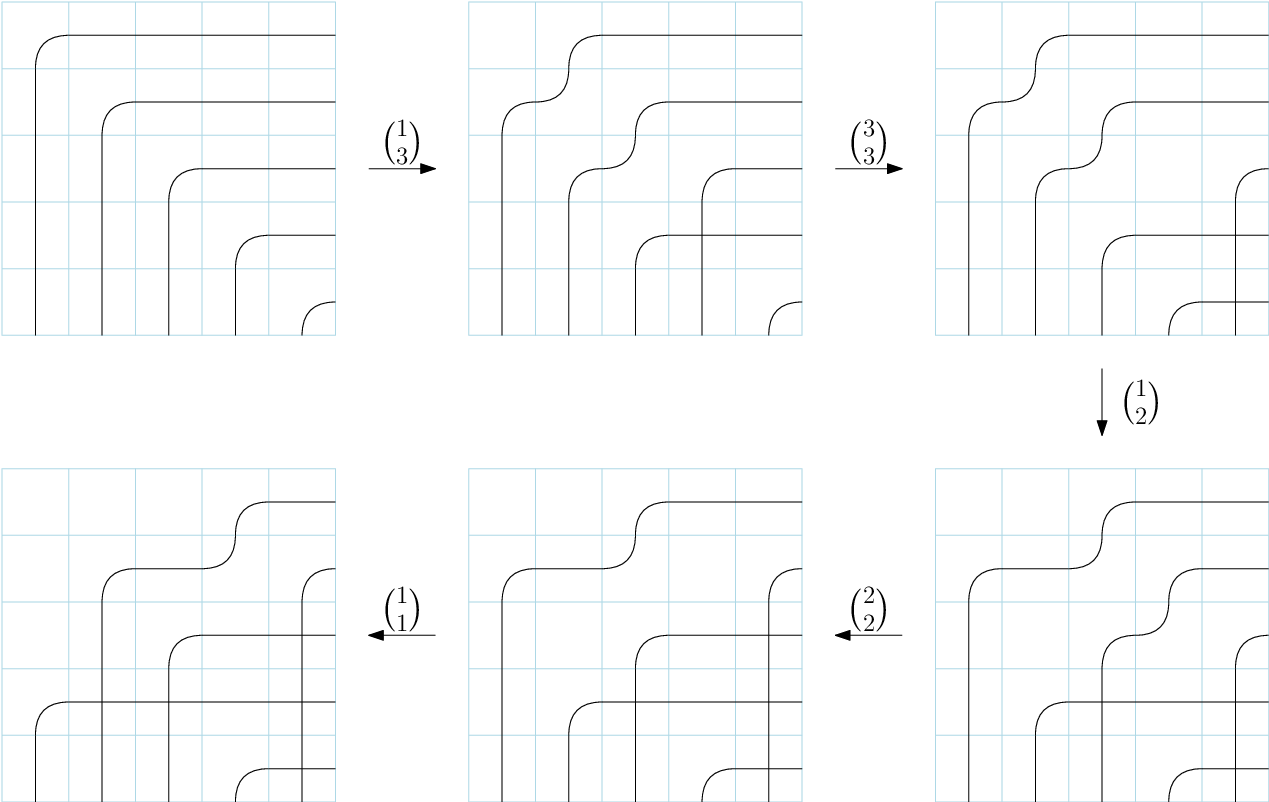}
        \caption{Insertion of $\left( \begin{smallmatrix}
        1 & 3 & 1 & 2 & 1 \\
        3 & 3 & 2 & 2 & 1 \\
    \end{smallmatrix} \right)$}
        \label{fig:growth-ex}
    \end{figure}

    On the other hand, removing all biletters $\binom{1}{k}$ in the original biword, we obtain the biword $\left( \begin{smallmatrix}
        3 & 2 \\
        3 & 2 \\
    \end{smallmatrix} \right)$. The BPD of this biword is shown in Figure \ref{fig:growth-ex2}.

    \begin{figure}[h!]
        \centering
        \includegraphics[scale = 0.4]{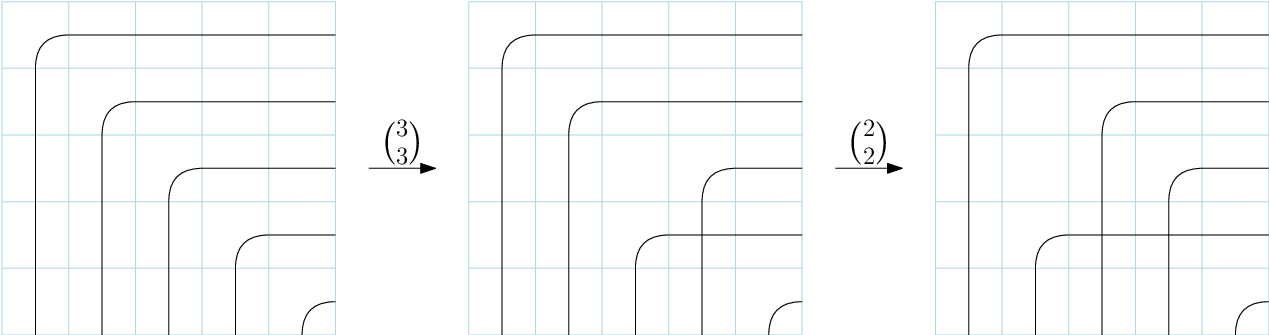}
        \caption{Insertion of $\left( \begin{smallmatrix}
        3 & 2 \\
        3 & 2 \\
    \end{smallmatrix} \right)$}
        \label{fig:growth-ex2}
    \end{figure}
    
    Let $D$ be the BPD corresponding to the original biword
    $\left( \begin{smallmatrix}
        1 & 3 & 1 & 2 & 1 \\
        3 & 3 & 2 & 2 & 1 \\
    \end{smallmatrix} \right)$
    (in Figure \ref{fig:growth-ex}), and $D'$ be the BPD corresponding to the new biword $\left( \begin{smallmatrix}
        3 & 2 \\
        3 & 2 \\
    \end{smallmatrix} \right)$ (in Figure \ref{fig:growth-ex2}). Theorem \ref{thm:jdt} says that $D' = \rect(D)$. This is indeed the case as shown in Figure \ref{fig:growth-ex3}.

    \begin{figure}[h!]
        \centering
        \includegraphics[scale = 0.4]{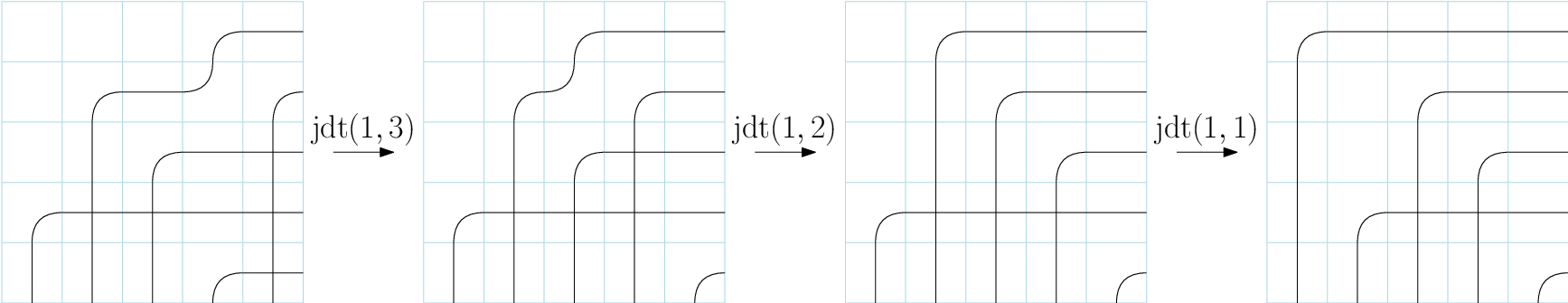}
        \caption{Rectification process}
        \label{fig:growth-ex3}
    \end{figure}

    \begin{defn}[\cite{BJS}]\label{def:compatible-sequence}
For a permutation $\pi$ with $\ell(\pi)=\ell$, a pair of integer sequences $\big(\mathbf{a}=(a_1,\ldots,a_{\ell}),\mathbf{r}=(r_1,\ldots,r_{\ell})\big)$ is a  \textbf{bounded reduced compatible sequence} of $\pi$ if
$s_{a_1}\cdots s_{a_\ell}$ is a reduced word of $\pi$,
 $r_1\leq\cdots \leq r_{\ell}$ is weakly increasing,
 $r_j\leq a_j$ for $j=1,\ldots,\ell$, and
 $r_j<r_{j+1}$ if $a_j<a_{j+1}$.

\end{defn}

The following theorem is an immediate consequence of Theorem~\ref{thm:jdt} and the definition of the PD-BPD bijection in \cite{gao2023canonical}.
\begin{thm}
    Let $Q:=\left( \begin{smallmatrix}
        b_1 & b_2 & \ldots & b_\ell \\
        k_1 & k_2 & \ldots & k_\ell \\
    \end{smallmatrix} \right)$
    and let $a = \max\{b_i~|~ 1\leq i\leq \ell\}$, and $(\pi_{i,j})_{0\le i\le a, 0\le j\le \ell}$ be the growth diagram of $Q$. Then the rightmost vertical chain
    \[\mathrm{id}=\pi_{a,\ell}\lessdot \cdots \lessdot \pi_{0,\ell}\] uniquely recovers a bounded reduced compatible sequence, and this bijects to $\varphi(Q)$ under the bijection in \cite{gao2023canonical}.

    Explicitly, by Corollary~\ref{cor:jdt-si}, for each $1\le i\le a$, we have $s_{i,m_i},\cdots s_{i,1}\pi_{i,\ell}=\pi_{i-1,\ell}$., where $s_{i,1}>\cdots >s_{i,m_i}$. Then the compatible sequence that corresponds to $Q$ is 
    \[\binom{\mathbf{a}}{\mathbf{r}}= \begin{pmatrix}
        s_{0,1} & \cdots & s_{0,m_1} & s_{1,1} & \cdots &s_{1,m_1} &\cdots & s_{a-1,1} & \cdots & s_{a-1,m_{a-1}} \\
        1 & \cdots & 1& 2 & \cdots &2 & \cdots & a & \cdots & a
    \end{pmatrix}. \]
\end{thm}
\begin{ex}
    Continuing Example~\ref{ex:growth}, the compatible sequence that corresponds to the chain
    \[12345\lessdot 12435 \lessdot 13425 \lessdot 25314\]
    is \[\binom{\mathbf{a}}{\mathbf{r}}= \begin{pmatrix}
        s_4 & s_3 & s_1 & s_2 & s_3 \\
        1 & 1 & 1 & 2 & 3 \\
    \end{pmatrix}. \]
\end{ex}

\section{Insertion and jeu de taquin on bumpless pipe dreams}\label{sec:insert-jdt}

    In this section we review the definition of right insertion and jeu de taquin on BPDs. We also define reverse jeu de taquin in this section, as well as the paths associated to these algorithms.
    \subsection{Right insertion}\label{subsec:def-insert}

    Recall the definition of right insertion in \cite{HP}.
    For the definition of basic operations $\mindroop$ and $\cbswap$ we refer the reader to \cite[Definition 3.1]{HP}.
    \begin{defn}[Right insertion]\label{def:insertion}
        Let $D\in \BPD(\pi)$ and $\binom{b}{k}$ be a biletter. We define the right insertion algorithm that produces $\left( D \leftarrow \binom{b}{k} \right)\in \BPD(\pi t_{\alpha,\beta})$ where $\alpha \le k < \beta$ and  $\ell(\pi)+1=\ell(\pi t_{\alpha,\beta})$ as follows. Let $(i, j)$ be the position of the $\rt$ on row $b$ with $j$ maximal.  
        \begin{enumerate}[(1)]
            \item  Perform a $\mindroop$ at $(i,j)$ and let $(i_1, j_1)=\droop(i,j)$. 
            \item If $(i_1,j_1)$ is a $\jt$, let $(i_1,j_2)$ be the $\rt$ with $j_2<j_1$ maximal. Update $(i,j)$ to be $(i_1,j_2)$ and go back to the beginning of step (1).
            \item If $(i_1,j_1)$ is a $\bt$, we check whether the two pipes passing through this tile have already crossed.
            \begin{enumerate}
                \item If yes, perform a $\cbswap$ and update $(i,j)$ to be $\swap(i_1,j_1)$, and go to step (1).
                \item Otherwise, we let $p$ denote the pipe of the $\rt$-turn of this tile and check if $p$ exits from row $r \le k$. If so, we update $(i,j)$ to be $(i_1,j_1)$ and go to step (1). Otherwise, we replace the tile at $(i_1,j_1)$ with a $\+$ and terminate the algorithm.
            \end{enumerate}
        \end{enumerate}
    \end{defn}

        When the input is a plactic biword, the algorithm simplifies. Recall the following Lemma about right insertion.

    \begin{lemma}(\cite[Lemma 5.2]{HP})\label{lem:droopwid}
       During right insertion of $D\leftarrow \binom{b}{k}$, $D\in\BPD(\pi)$ where $k$ is at most the first descent of $\pi$,  every $\mindroop$ is performed on a pipe $p$ such that $p$ does not contain any horizontal segment in a $\+$. In particular, this means that every $\mindroop$ is bounded by a width 2 rectangle. 
    \end{lemma}

    Therefore, we can define $\maxdroop$ as a maximal sequence of consecutive $\mindroop$s at tiles in a same column. By Lemma \ref{lem:droopwid}, every $\maxdroop$ is also bounded by a width 2 rectangle. Furthermore, we can consider insertion as a sequence of $\maxdroop$ and $\cbswap$, followed by a terminal move that replaces a $\bt$ with a $\+$. Thus, we can define the insertion path as follows.

    \begin{defn}[Insertion path]\label{def:insert-path}
        Let $D\in \BPD(\pi)$ and $\binom{b}{k}$ be a biletter. We define the insertion path of $\left(D \leftarrow \binom{b}{k}\right)$ as a sequence of squares $(i_0,j_0),(i_1,j_1),\ldots,(i_\ell,j_\ell)$ such that
        \begin{enumerate}
            \item $(i_0,j_0)$ is the starting point of $\left(D \leftarrow \binom{b}{k} \right)$, i.e. $(i_0,j_0)$ is the position of the $\rt$ on row $b$ with $j$ maximal.
            \item If $(i_k,j_k)$ is a $\jt$, let $(i_k,j_k')$ be the $\rt$ with $j_k'<j_k$ maximal, then $(i_{k+1},j_{k+1}) := (i_k,j_k')$.
            \item Otherwise, if the algorithm performs a $\maxdroop$ at $(i_k,j_k)$ then $(i_{k+1},j_{k+1}) := \maxd(i_k,j_k)$. If it performs a $\cbswap$ at $(i_k,j_k)$ then $(i_{k+1},j_{k+1}) := \swap(i_k,j_k)$.
            \item The insertion algorithm performs the terminal move at $(i_\ell,j_\ell)$.
        \end{enumerate}

        In addition, we say \textbf{the insertion path goes through pipe $p$} if it performs a $\maxdroop$ at pipe $p$, or its terminal move involves pipe $p$.
    \end{defn}

    \begin{figure}[h!]
        \centering
        \includegraphics[scale = 0.4]{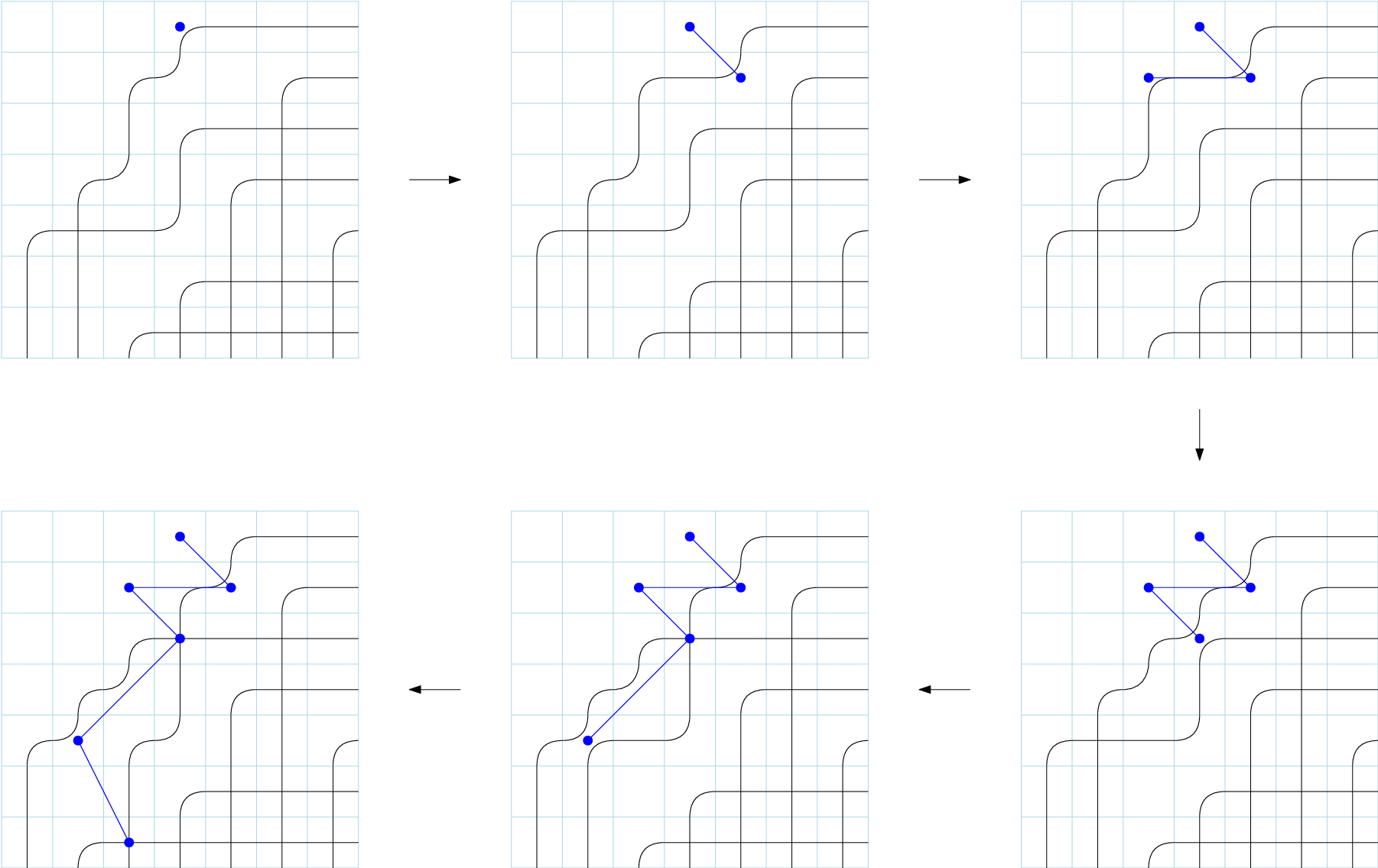}
        \caption{Insertion path}
        \label{fig:ins-path-ex}
    \end{figure}

    For example, Figure \ref{fig:ins-path-ex} shows the insertion path of $\left(D \leftarrow \binom{1}{3}\right)$ where $D$ is the top left BPD. The  dots form the sequence of squares in the insertion path. This insertion path goes through pipes $2$ and $3$.

    \subsection{Jeu de taquin}

    Recall the definition of jeu de taquin in \cite{gao2023canonical}.

    \begin{defn}[Jeu de taquin]\label{def:jdt}
        Given $D\in\BPD(\pi)$ with $\ell(\pi)>0$, the following process produces another bumpless pipe dream $\nabla D\in\BPD(\pi')$ where $\ell(\pi')=\ell(\pi)-1$. Let $r$ be the smallest row index such that the row $r$ of $D$ contains \bl-tiles. To initialize, mark the rightmost \bl-tile in row $r$ with a label ``$\times$".
        \begin{enumerate}
            \item If the marked \bl-tile is not the rightmost \bl-tile in a contiguous block of \bl-tiles in its row, move the label ``$\times$" to the rightmost \bl-tile of this block. Assume the marked tile has coordinate $(x,y)$, and the pipe going through $(x,y+1)$ is $p$. 
            \item If $p\neq y+1$, suppose the \jt-tile of $p$ in column $y+1$ has coordinate $(x',y+1)$ for some $x'>x$. Let $U$ be the rectangle with NW corner $(x,y)$ and SE corner $(x',y+1)$. We modify the tiles in $U$ as follows:
        
            \begin{enumerate}
            
                \item For each pipe $q$ intersecting $p$ at some $(z,y+1)$ where $x<z<x'$ and $(z,y)$ is an \rt-tile, let $(z',y)$ be the \jt-tile of $q$ in column $y$. Ignoring the presence of $p$, droop $q$ at $(z,y)$ within $U$, so that $(z,y+1)$ becomes an \rt-tile and $(z',y+1)$ becomes a \jt-tile;
                \item Undroop pipe $p$ at $(x',y+1)$ into $(x,y)$, and move the mark to $(x',y+1)$.
            \end{enumerate}
            We call this process $\recundroop(x,y)$, and define $\recund(x,y) := (x',y+1)$. Now go back to step (1) and repeat.
            \item If $p=y+1$, the pipes $y$ and $y+1$ must intersect at some $(x',y+1)$ for some $x'>x$. Replace  this \+-tile with a \bt-tile, undroop the \jt-turn of this tile into $(x,y)$ and adjust the pipes between row $x$ and $x'$  in a same fashion as described in  Step (2) above. We call this process $\uncross(x,y)$. We are done after this step.
        
        \end{enumerate}
        Let $a$ be the column $y$ in Step (3), then the final BPD is of the permutation $s_a\pi$. Note that since pipes $y$ and $y+1$ in Step (3) intersect, we do have $\ell(s_a\pi) = \ell(\pi)-1$. Finally, let $(b,r)$ be the starting point of jeu de taquin, that is, the first marked square before Step (1). We will sometimes use the notation $\jdt(b,r)$ instead of $\nabla$ to emphasize that jeu de taquin starts from $(b,r)$.

        In addition, if this process start at row $r$ and ends by swapping pipes $a$ and $a+1$, we define $\pop(D) := (a,r)$.
    \end{defn}

    Similar to insertion, we can define the jeu de taquin path as follows.

    \begin{defn}[Jeu de taquin path]\label{def:jdt-path}
        Let $D\in \BPD(\pi)$, we define the jeu de taquin path of $D$ as a sequence of squares $(i_0,j_0),(i_1,j_1),\ldots,(i_\ell,j_\ell)$ such that
        \begin{enumerate}
            \item $(i_0,j_0)$ is the starting point of jeu de taquin.
            \item Suppose the mark is at $(i_k,j_k)$, if $(i_k,j_k+1)$ is a $\bl$, then $(i_{k+1},j_{k+1}) := (i_k,j_k+1)$.
            \item Otherwise, $(i_k,j_k+1)$ is not a $\bl$. If the pipe at $(i_k,j_k+1)$ is not $j_k+1$, then jeu de taquin performs a $\recundroop(i_k,j_k)$, so $(i_{k+1},j_{k+1}) := \recund(i_k,j_k)$.
            \item Jeu de taquin performs an $\uncross$ at $(i_\ell,j_\ell)$, terminating the process.
        \end{enumerate}
    \end{defn}

    For example, Figure \ref{fig:jdt-path-ex} shows the jeu de taquin path of $\nabla(D)$ where $D$ is the top left BPD. The red dots form the sequence of squares in the jeu de taquin path. The process starts from row 1 and terminates by swapping pipes 5 and 6, so we have $\pop(D) = (5,1)$.

    \begin{figure}[h!]
        \centering
        \includegraphics[scale = 0.4]{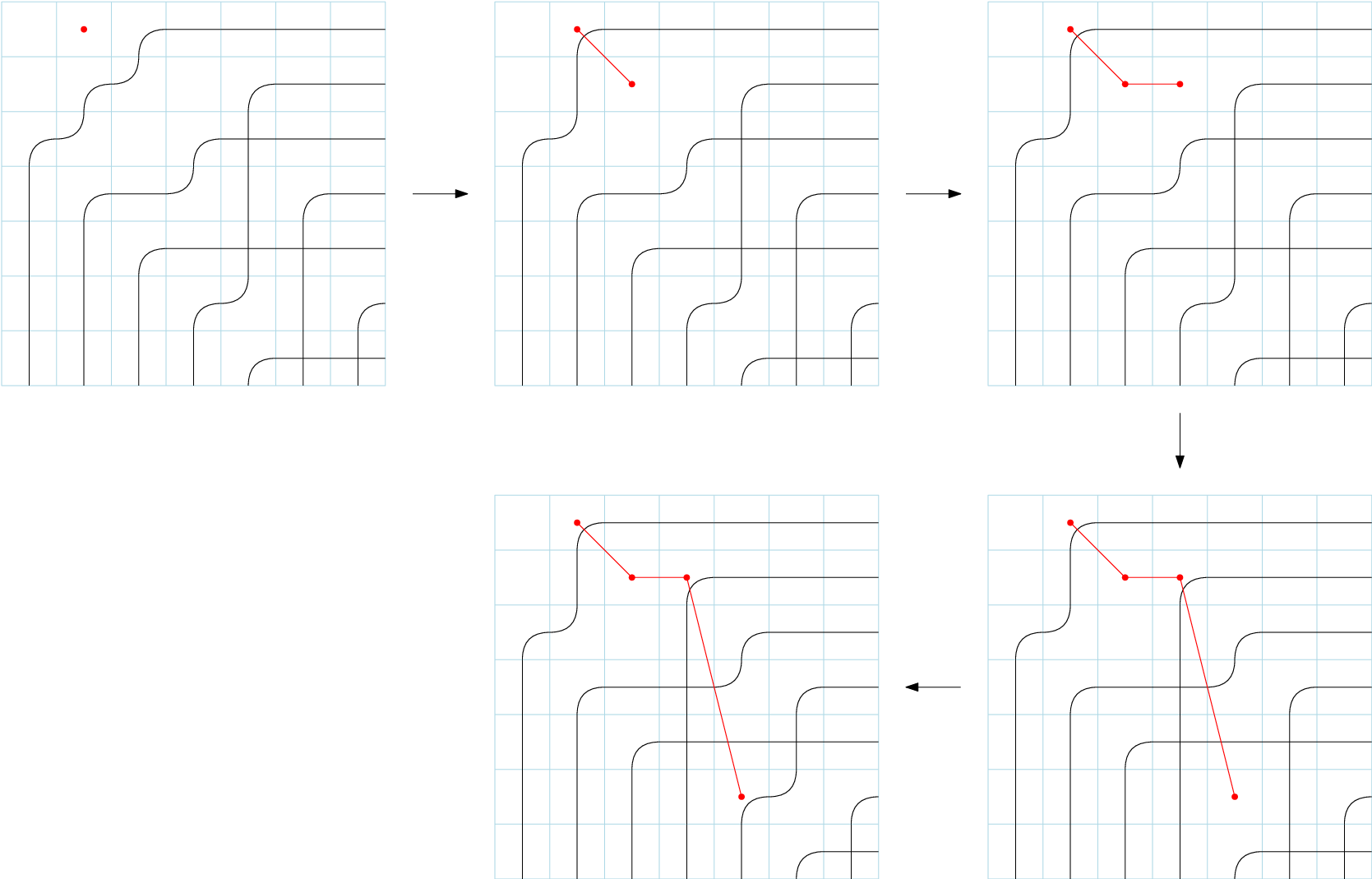}
        \caption{Jeu de taquin path}
        \label{fig:jdt-path-ex}
    \end{figure}

    We can also define a reversed jeu de taquin move as the inverse of jeu de taquin as follows.

    \begin{defn}[Reversed jeu de taquin]\label{def:rev-jdt}
        Given $D\in\BPD(\pi)$, and $s_a$ is not a left-descent of $\pi$ the following process produces another bumpless pipe dream $\Delta_{r,a} D\in\BPD(s_a\pi')$ where $\ell(s_a\pi)=\ell(\pi)+1$ when the execution is successful. Let $(x,y)$ be the only $\rt$ of pipe $a$ in column $a$, and $(x',y+1)$ be the only $\rt$ of pipe $a+1$ in column $a+1$. Droop pipe $a$ at $(x,y)$ into $(x',y+1)$ in a same fashion as in Step (2) of Definition \ref{def:jdt}. Replace the $\bt$ at $(x',y+1)$ with a $\+$ and mark $(x,y)$ with a label ``$\times$". We call this process $\cross(x,y)$.
        \begin{enumerate}
            \item Assume the marked tile has coordinate $(x,y)$, if $x\leq r$, terminate the process.
            \item Otherwise, if the marked \bl-tile is not the leftmost \bl-tile in a contiguous block of \bl-tiles in its row, move the label ``$\times$" to the leftmost \bl-tile of this block. If the marked tile ends up in column 1, the algorithm fails and does not produce an output.
            \item Assume the marked tile has coordinate $(x,y)$, and the pipe going through $(x,y-1)$ is $p$. Suppose the $\rt$-tile of $p$ in column $y-1$ has coordinate $(x',y-1)$, then we droop pipe $p$ at $(x',y-1)$ into $(x,y)$ in a same fashion as in Step (2) of Definition \ref{def:jdt}. We call this process $\recdroop(x,y)$, and define $\recd(x,y) := (x',y-1)$. Now go back to Step (1) and repeat.  
        \end{enumerate}
    \end{defn}

    We can also define the reversed jeu de taquin path as follows for successful executions of reverse jeu de taquin.

    \begin{defn}[Reversed jeu de taquin path]\label{def:rev-jdt-path}
        We define the reversed jeu de taquin path as a sequence of squares $(i_0,j_0),(i_1,j_1),\ldots,(i_\ell,j_\ell)$ such that
        \begin{enumerate}
            \item $(i_0,j_0)$ is the starting point of reversed jeu de taquin.
            \item If the process does not terminate, suppose the mark is at $(i_k,j_k)$. If $(i_k,j_k-1)$ is a $\bl$, then $(i_{k+1},j_{k+1}) := (i_k,j_k-1)$.
            \item Otherwise, $(i_k,j_k-1)$ is not a $\bl$, so reversed jeu de taquin performs a $\recdroop(i_k,j_k)$, then $(i_{k+1},j_{k+1}) := \recd(i_k,j_k)$.
            \item Reversed jeu de taquin terminates at $(i_\ell,j_\ell)$.
        \end{enumerate}
    \end{defn}

    For example, Figure \ref{fig:rjdt-path-ex} shows the reversed jeu de taquin path of $\Delta_{1,5}(D)$ where $D$ is the top left BPD. The red dots form the sequence of squares in the reversed jeu de taquin path. One can check that this is exactly the reversed process of Figure \ref{fig:jdt-path-ex}.

    \begin{figure}[h!]
        \centering
        \includegraphics[scale = 0.4]{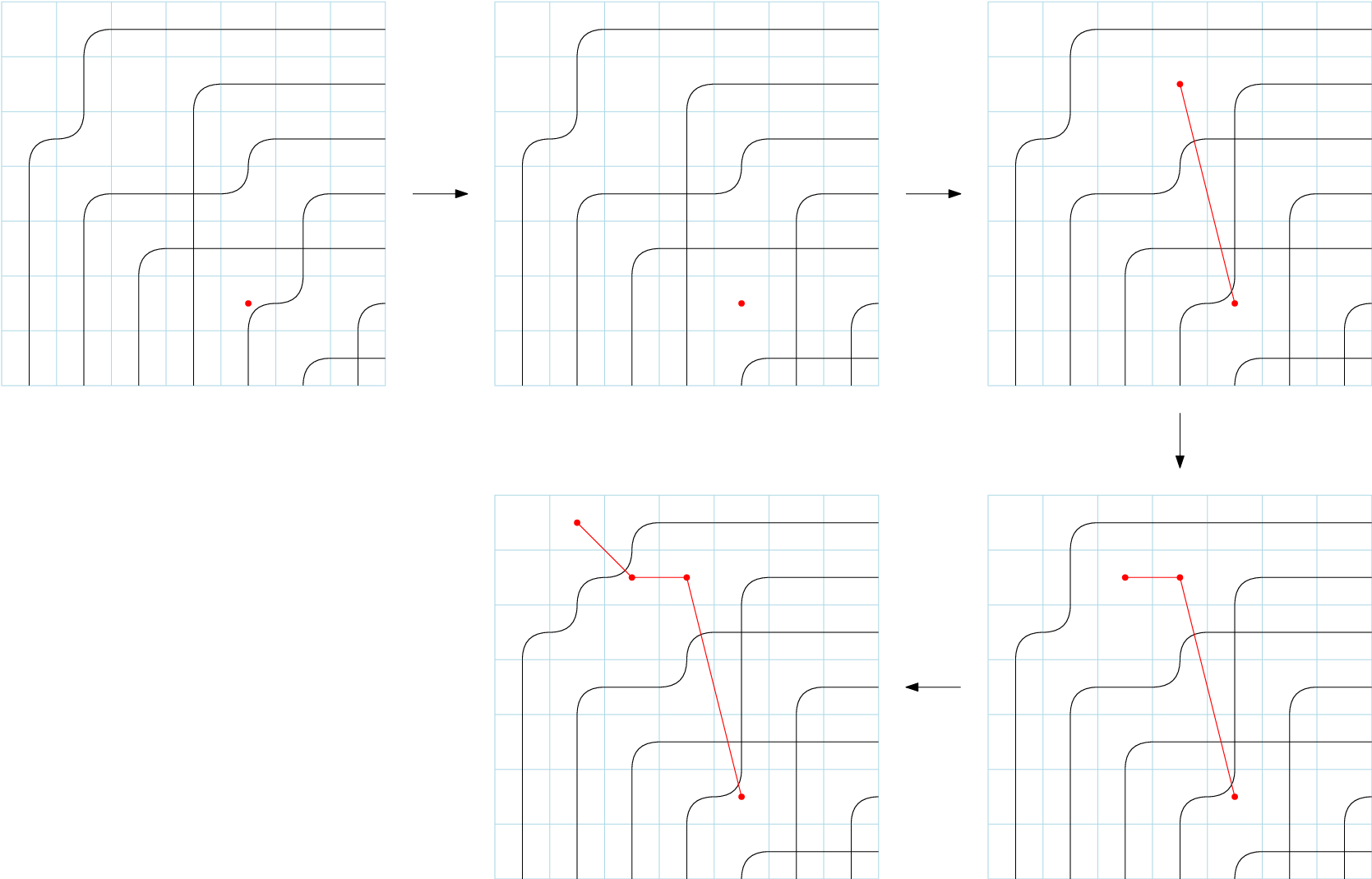}
        \caption{Reversed jeu de taquin path}
        \label{fig:rjdt-path-ex}
    \end{figure}

\section{Proofs of growth rules}

    \begin{lemma}\label{lem:rule-1ab}
        With the same notation as in Theorem \ref{thm:local-rule}, rules (1a) and (1b) of Theorem \ref{thm:local-rule} are true. In these cases, if $D_\pi = \rect(D_\mu)$ then $D_\sigma = \rect(D_\rho)$.
    \end{lemma}

    \begin{proof}
        The fact that (1a) and (1b) are true follow directly from the definition of growth diagram. The second statement is also immediate.
    \end{proof}

    Before proving Theorems \ref{thm:jdt} and \ref{thm:local-rule} for the remaining two rules, we will prove a few  lemmas.

    \begin{lemma}\label{lem:consec-pipe}
        Suppose during insertion $\left( D \leftarrow \binom{b}{k} \right)$, the insertion path is on the $\rt$-tile in column $p$ of pipe $p$, then from this point, the insertion path will go through pipes $p+1,p+2,\ldots$ until it stops.
    \end{lemma}

    \begin{proof}
        Let $(x,p)$ be the coordinate of the $\rt$-tile in column $p$ of pipe $p$, and let $(x',p+1) := \maxd(x,p)$. After $\maxdroop(x,p)$, $(x',p+1)$ cannot be a $\jt$. This is because if $(x',p+1)$ is a $\jt$, the insertion path will go left to $(x',p)$ and perform another $\maxdroop$ on $(x',p)$. This means that the insertion path performs 2 consecutive $\maxdroop$ in the same column, which contradicts the definition of $\maxdroop$. Thus, after $\maxdroop(x,p)$, $(x',p+1)$ is a $\bt$. Let $q$ be the pipe of the $\rt$ in $(x',p+1)$, clearly, $q\leq p+1$. If $q<p$, then $p$ and $q$ must cross at some square $(y,p)$, so the insertion path will perform a $\cbswap(x',p+1)$ followed by a $\maxdroop(y,p)$. Once again, this contradicts the definition of $\maxdroop$. Thus, $q = p+1$, and the insertion path moves from pipe $p$ to $p+1$. Note that now $(x',p+1)$ is the $\rt$-tile in column $p+1$ of pipe $p+1$, so we can repeat the same argument. Hence, the insertion path will go through pipes $p+1,p+2,\ldots$ until it stops.
    \end{proof}

    The next lemma is the key lemma to prove rules (1c) and (2).

    \begin{lemma}\label{lem:jdt-insert}
        Let $D\in\BPD(\pi)$, and $D' = \nabla(D)$. Suppose $\pop(D) = (i,c)$, then by definition $D'\in\BPD(\sigma)$ where $\sigma = s_i\pi$. Given $b\geq c$ and $k$ such that the smallest descent in $\pi$ is at least $k$. Suppose the insertion path of $D' \leftarrow \binom{b}{k}$ goes through pipes $p_1<p_2<\ldots<p_\ell$. Let $P := \{p_1,p_2,\ldots,p_\ell\}$, then
        \begin{enumerate}
            \item if $i = p_{j}$ and $i+1 \neq p_{j+1}$ for some $1\leq j\leq \ell-1$, then $D \leftarrow \binom{b}{k}$ goes through pipes $p_1,\ldots,p_{j-1},p_j+1,p_{j+1},\ldots,p_\ell$;
            % \daoji{Do you have an example where $p_{j+1},\cdots$ are non-empty?}
            %
            \item if $i = p_{\ell-1}$ and $i+1 = p_\ell$, then $D \leftarrow \binom{b}{k}$ goes through pipes $p_1,\ldots,p_{\ell-2},p_\ell,p_\ell + 1,p_\ell+2,\ldots$ until it terminates;
            \item if $i = p_\ell$ then $D \leftarrow \binom{b}{k}$ goes through pipes $p_1,\ldots,p_{\ell-1},p_\ell + 1$;
            \item otherwise, $D \leftarrow \binom{b}{k}$ goes through pipes $p_1,\ldots,p_\ell$.
        \end{enumerate}
        In particular, unless $i = p_{\ell-1}$ or $i = p_{\ell}$, the last two pipes of $D \leftarrow \binom{b}{k}$ are still $p_{\ell-1}$ and $p_\ell$.
    \end{lemma}

    Before proving Lemma \ref{lem:jdt-insert}, let us give some examples of the cases. In Figure \ref{fig:jdt-insert-ex12}, we have a BPD $D$ with $\pop(D) = (3,1)$. In $D' = \nabla(D)$, the insertion path of $D' \leftarrow \binom{2}{5}$ goes through pipes $2,3,5,6,7$. Since $i = 3$ is one of the pipes, but $i+1 = 4$ is not, the insertion path of $D \leftarrow \binom{2}{5}$ goes through pipes $2,4,5,6,7$. This is case (1) of Lemma \ref{lem:jdt-insert}. On the other hand, the insertion path of $D' \leftarrow \binom{1}{5}$ goes through pipes $1,2,3,4$. Since $i$ and $i+1$ are the last two pipes, the insertion path of $D \leftarrow \binom{1}{5}$ goes through pipes $1,2,4,5,6,7$. This is case (2) in Lemma \ref{lem:jdt-insert}. Finally, the insertion path of $D' \leftarrow \binom{2}{2}$ goes through pipes $2$ and $3$. Thus, the insertion path of $D \leftarrow \binom{2}{2}$ goes through pipes $2$ and $4$. This is case (3) in Lemma \ref{lem:jdt-insert}.

    \begin{figure}[h!]
        \centering
        \includegraphics[scale = 0.4]{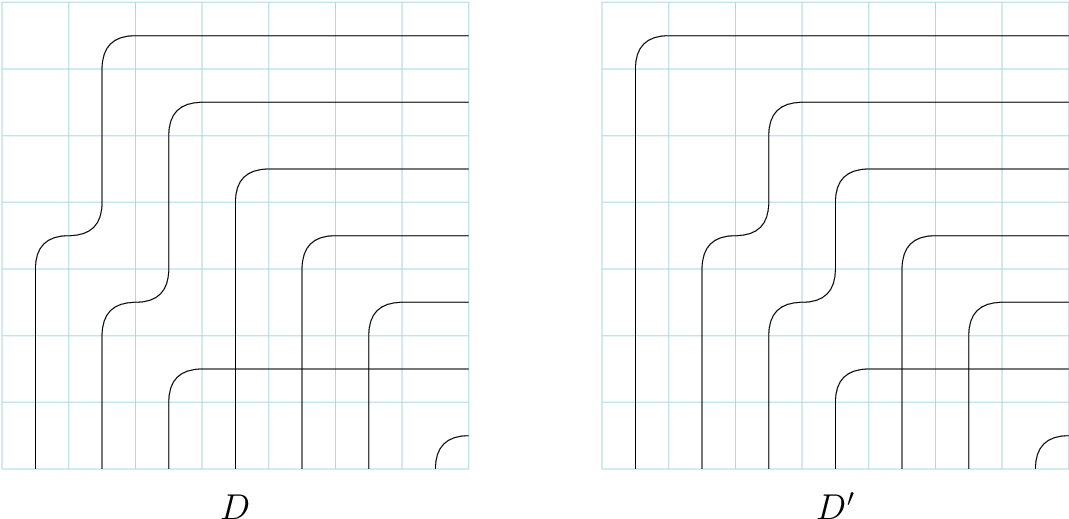}
        \caption{$D$ (left) and $D' = \nabla(D)$ (right)}
        \label{fig:jdt-insert-ex12}
    \end{figure}
    
    % \daoji{What about the possibility that $i$ does not show up in the path? }
    \begin{proof}[Proof of Lemma \ref{lem:jdt-insert}]
        Since $\pop(D) = (i,c)$, $c$ is the smallest row with a $\bl$ in $D$, and $D = \Delta_{c,i}(D')$. Let us study the interaction between the reversed $\jdt$ path $\Delta_{c,i}(D')$ (denoted $\mathcal{J}$) and the insertion path $D' \leftarrow \binom{b}{k}$ (denoted $\mathcal{I}$). 

        \textbf{Case 1:} If $\mathcal{J}$ and $\mathcal{I}$ do not share any common points, then the only way that $\mathcal{J}$ can affect $\mathcal{I}$ is at the beginning where $\mathcal{J}$ performs a $\cross(x,i)$. This crosses pipes $i$ and $i+1$, and we have a few cases.

        \begin{enumerate}
            \item[(1a)] $i\neq p_j, p_j-1$ for all $1\leq j\leq \ell$. This means that $i,i+1\notin P$, so crossing pipes $i$ and $i+1$ does not affect $\mathcal{I}$.

            \item[(1b)] $i = p_j - 1$ for some $1\leq j < \ell$. This means that $\sigma^{-1}(p_j-1) < \sigma^{-1}(p_j)$. Since $j<\ell$, we must have $\sigma^{-1}(p_j)\leq k$. Thus, $\sigma^{-1}(p_j-1) < \sigma^{-1}(p_j)\leq k$. However, we have $\pi = s_i\sigma$; this means that $\pi$ has a descent before position $k$. This is not possible.

            \item[(1c)] $i = p_j$ and $i+1\neq p_{j+1}$ for some $1\leq j \leq \ell-1$. Note that since $\mathcal{J}$ and $\mathcal{I}$ do not share any common point, in $D'$, $(x,i)$ is not on the same column as the point of pipe $i$ on $\mathcal{I}$. Thus, crossing pipe $i$ and $i+1$ does not affect $\mathcal{I}$, and the new insertion path of $D\leftarrow \binom{b}{k}$ simply goes through pipes $\ldots,p_{j-1}, p_j +1, p_{j+1},\ldots$. This is case (1) in Lemma \ref{lem:jdt-insert}.

            \item[(1d)] $i = p_j$, and $i+1 = p_{j+1}$ for some $1\leq j < \ell-1$. By the same argument as in case (1b), we have $\sigma^{-1}(p_j) < \sigma^{-1}(p_{j+1})\leq k$, so in $\pi$, there is a descent before position $k$. This is also not possible.

            \item[(1e)] $i = p_{\ell - 1}$ and $i+1 = p_\ell$. Again, since $\mathcal{J}$ and $\mathcal{I}$ do not share any common point, in $D'$, $(x,i)$ is not on the same column as the point of pipe $i$ on $\mathcal{I}$. After $\cross(x,i)$ crosses pipes $p_{\ell - 1}$ and $p_\ell$, the new insertion path does not go through pipe $p_{\ell-1}$ anymore. Instead, it will perform a $\maxdroop$ followed by a $\cbswap$ and stay in pipe $p_\ell$ (see Figure \ref{fig:jdt-insert-1d}). Furthermore, since the intersection of pipes $p_{\ell-1}$ and $p_\ell$ is now in column $p_\ell$, the new insertion path is now on a $\rt$ of pipe $p_\ell$ in column $p_\ell$. By Lemma \ref{lem:consec-pipe}, the new insertion path will go through pipes $p_\ell+1,p_\ell+2,\ldots$ until it terminates. This is case (2) in Lemma \ref{lem:jdt-insert}.

            \begin{figure}[h!]
                \centering
                \includegraphics[scale = 0.4]{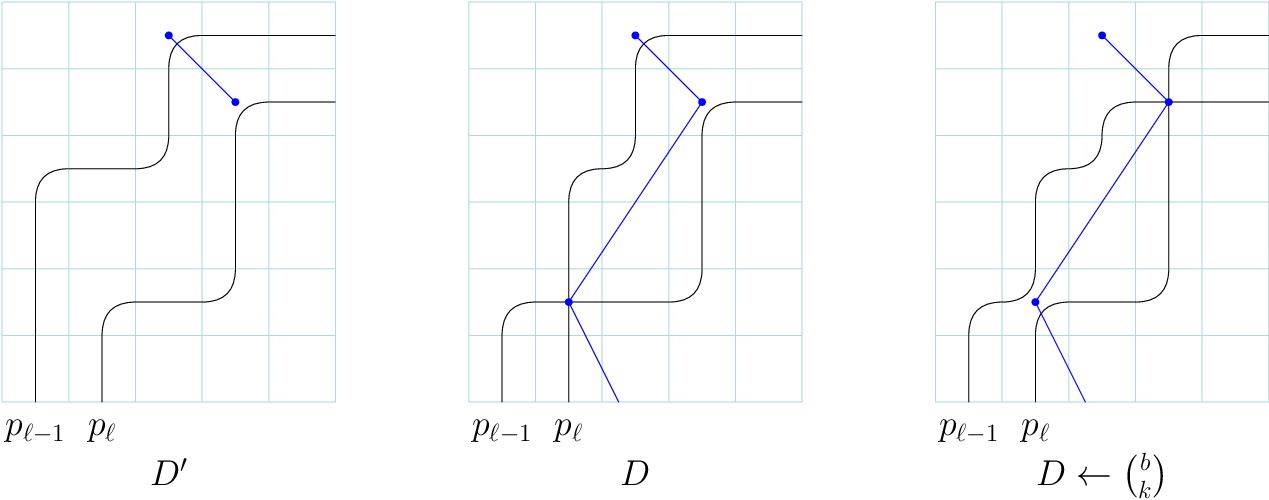}
                \caption{Case (1e)}
                \label{fig:jdt-insert-1d}
            \end{figure}       

            \item[(1f)] $p_{\ell-1}< i$ and $i+1 = p_\ell$. Note that $\mathcal{I}$ terminates by crossing pipe $p_{\ell-1}$ and $p_\ell$, in $\sigma$, every value between $p_{\ell-1}$ and $p_\ell$ has to be either greater than $p_{\ell}$ or smaller than $p_{\ell-1}$. This means that we cannot have $\sigma^{-1}(p_{\ell}-1) < \sigma^{-1}(i) < \sigma^{-1}(p_\ell)$. If $\sigma^{-1}(i)<\sigma^{-1}(p_{\ell-1})$, then there is a descent before $\sigma^{-1}(p_{\ell-1}) \leq k$, which is not possible. If $\sigma^{-1}(p_\ell) = \sigma^{-1}(i+1) < \sigma^{-1}(i)$, then it is impossible to perform a $\cross$ on pipe $i$. Hence, this case is not possible.

            \item[(1g)] $i = p_\ell$. Again, since $\mathcal{J}$ and $\mathcal{I}$ do not share any common point, in $D'$, $(x,i)$ is not on the same column as the point of pipe $i$ on $\mathcal{I}$. Thus, crossing pipe $i$ and $i+1$ does not affect $\mathcal{I}$, except the new insertion path now go through pipes $p_1, p_2, \ldots, p_{\ell-1}, p_\ell + 1$. This is case (3) in Lemma \ref{lem:jdt-insert}.
        \end{enumerate}
        
        \textbf{Case 2:} $\mathcal{J}$ and $\mathcal{I}$ share some common points. Observe that for $\mathcal{I}$, except for the $\maxdroop$ steps, all other steps move the path in the SW direction. On the other hand, for $\mathcal{J}$, all  steps move the path in the NW direction. Thus, if $\mathcal{J}$ and $\mathcal{I}$ share a sequence of points $(u_1,v_1),\ldots,(u_r,v_r)$, these points have to form a sequence of $\maxdroop$s in $\mathcal{I}$ and hence a sequence of $\recdroop$s in $\mathcal{J}$. In other words, we must have $(u_{j+1},v_{j+1}) = \maxd(u_j,v_j)$, and $(u_j,v_j) = \recd(u_{j+1},v_{j+1})$ for all $1\leq j < r$. Therefore, we have $v_j = v_1+(j-1)$. Furthermore, $(u_{j+1},v_{j+1}) = \maxd(u_j,v_j)$ for all $1\leq j < r$ means that each point $(u_j,v_j)$ for $2\leq j < r$ is a $\rt$. In addition, $\mathcal{J}$ goes through $(u_1,v_1)$ and $\mathcal{I}$ performs a $\maxdroop$ here, so $(u_1,v_1)$ is also a $\rt$. Thus, each point $(u_j,v_j)$, for $1\leq j < r$, is a $\rt$ of some pipe $\alpha_j$. Note that $\mathcal{I}$ goes through the pipes $\alpha_1,\ldots,\alpha_{r-1}$, so $\alpha_i$ and $\alpha_{i+1}$ cannot intersect for all $1\leq i < r-1$. We will first prove that in $D$, the new insertion path $D\leftarrow \binom{b}{k}$ goes through the $\rt$-tiles of pipes $\alpha_1,\ldots,\alpha_{r-1}$ in columns $v_1+1,\ldots,v_{r-1}+1$.

        First, we look at $(u_1,v_1)$.  There are only a few cases. If $\mathcal{J}$ performs a $\recdroop$ at $(u_1,v_1)$, then we claim that there is either no point before $(u_1,v_1)$ in $\mathcal{I}$, or the previous point is at $(u_1,v_1+e)$ for some $e$. This is because otherwise, $\mathcal{I}$ goes to $(u_1,v_1)$ after a $\maxdroop$, so the previous point is exactly $\recd(u_1,v_1)$ and $\mathcal{J}$ also goes through this point, which means $(u_1,v_1)$ is not the first point of the intersection of $\mathcal{I}$ and $\mathcal{J}$.
        % Similarly, if $\mathcal{J}$ moves to the left from $(u_1,v_1)$, then this means that $(u_1,v_1-1)$ is a $\bl$. Thus, there is no way $\mathcal{I}$ can go to $(u_1,v_1)$ after a $\maxdroop$, so again there is either no previous point on $\mathcal{I}$, or the previous point is at $(u_1,v_1+e)$ for some $e$. Either way, there are only two options: there is either no previous point in $\mathcal{I}$, i.e. $\mathcal{I}$ starts at $(u_1,v_1)$, or the previous point is at $(u_1,v_1+e)$ for some $e$.
        Now after $\mathcal{J}$ performs a $\recdroop$ that droops pipe $\alpha_1$ in $(u_1,v_1)$ down to $(u_2,v_2)$, pipe $\alpha_1$ has a $\rt$-tile in column $v_2 = v_1+1$. In either case, the new insertion path goes through this $\rt$-tile (see Figure \ref{fig:jdt-insert-2a}).

        \begin{figure}[h!]
            \centering
            \begin{subfigure}[b]{0.4\textwidth}
                \centering
                \includegraphics[scale = 0.4]{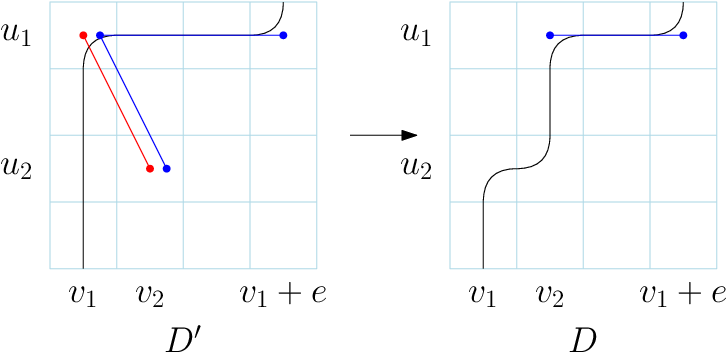}
                \caption{$\mathcal{I}$ does not start from $(u_1,v_1)$}
                \label{fig:jdt-insert-2a-fig1}
            \end{subfigure}
            \begin{subfigure}[b]{0.4\textwidth}
                \centering
                \includegraphics[scale = 0.4]{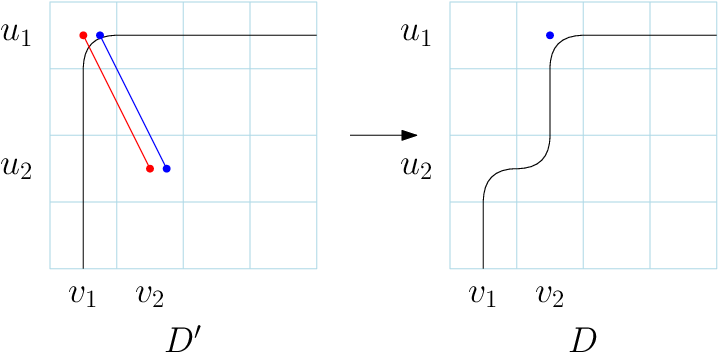}
                \caption{$\mathcal{I}$ starts from $(u_1,v_1)$}
                \label{fig:jdt-insert-2a-fig2}
            \end{subfigure}
            \caption{$\mathcal{I}$ at $(u_1,v_1)$}
            \label{fig:jdt-insert-2a}
        \end{figure}

        Now we apply the same argument for $(u_2,v_2),\ldots,(u_{r-1},v_{r-1})$. The key observation is that in $D$, the $\rt$-tile of pipe $\alpha_i$ in column $v_i+1$ is below the $\rt$-tile of pipe $\alpha_{i-1}$ in column $v_i$ since otherwise pipes $\alpha_{i-1}$ and $\alpha_i$ intersect. However, it is not below the $\jt$-tile of pipe $\alpha_{i-1}$ in column $v_i$ since $\mathcal{J}$ performs a $\recdroop$ that droops pipe $\alpha_{i-1}$ down to $(u_i,v_i)$. Thus, the new insertion path goes through exactly the $\rt$-tiles of pipes $\alpha_1,\ldots,\alpha_{r-1}$ in columns $v_1+1,\ldots,v_{r-1}+1$ (see Figure \ref{fig:jdt-insert-2b}).

        \begin{figure}[h!]
            \centering
            \includegraphics[scale = 0.4]{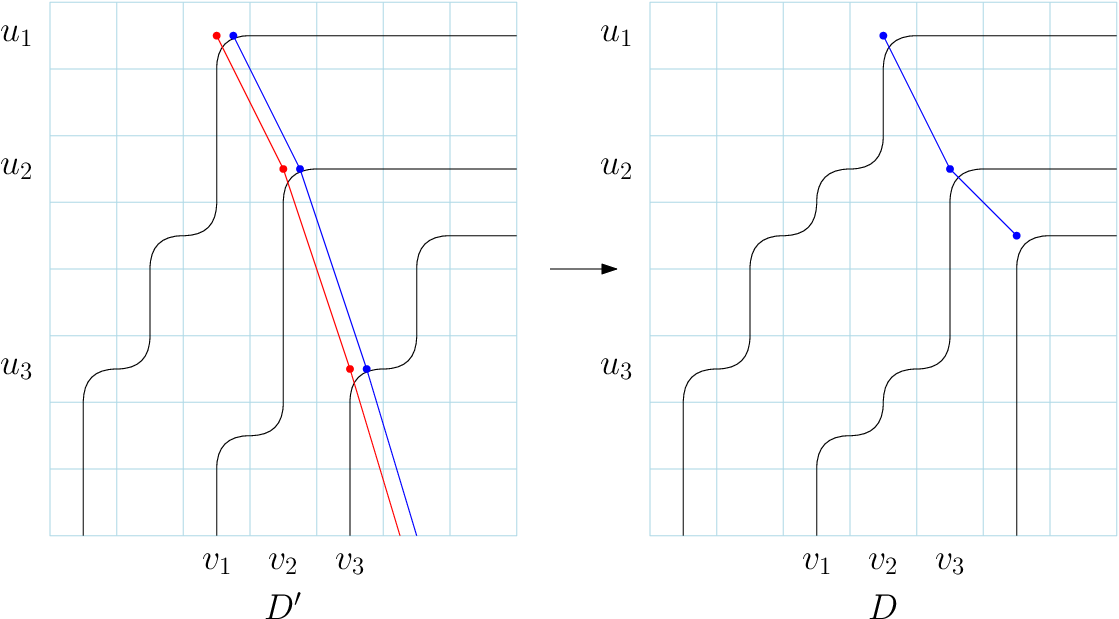}
            \caption{Segments of $D'\leftarrow\binom{b}{k}$ and $D\leftarrow\binom{b}{k}$ that go through the same pipes}
            \label{fig:jdt-insert-2b}
        \end{figure}

        Finally, we look at $(u_r,v_r)$. There are a few cases:

        \quad\textit{Case 2a:} $(u_r,v_r)$ is not the starting point of $\mathcal{J}$. Thus, there is a point on $\mathcal{J}$ prior to $(u_r,v_r)$. Since $(u_r,v_r)$ is a $\rt$, the previous point cannot be $(u_r,v_r+1)$. Thus, it has to be $(u, v_r+1)$ and $(u_r,v_r) = \recd(u,v_r+1)$. By the same argument above, the new insertion path of $D\leftarrow \binom{b}{k}$ goes through the $\rt$-tile of pipe $\alpha_r$ in column $v_r+1$. Let $(u_r', v_r+1)$ be the coordinate of this tile, we now show that after $(u_r',v_r+1)$, this new insertion path goes back to the next point on $\mathcal{I}$. First, note that since $(u_r,v_r)$ is a $\rt$, $\mathcal{I}$ performs a $\maxdroop$ here. Let $(v,v_r+1) := \maxd(u_r,v_r)$, we claim that $v>u$. This is because $(u_r,v_r) = \recd(u,v_r+1)$, so every $\rt$-tile between $(u,v_r+1)$ and $(u_r,v_r+1)$ intersects pipe $\alpha_r$ again. Also, note that $(u,v_r+1)$ cannot be the starting point of $\mathcal{J}$ since otherwise $(v, v_r+1)$ cannot exist. Now we consider the point prior to $(u,v_r+1)$ on $\mathcal{J}$.

        \begin{itemize}
            \item If it is $(u, v_r+2)$, then this square is a $\bl$. Thus, in $D\leftarrow \binom{b}{k}$, the new insertion path droops from $(u_r',v_r+1)$ down to $(u,v_r+2)$, then moves left to $(u,v_r)$, and finally droops down to $(v,v_r+1)$ (see Figure \ref{fig:jdt-insert-2c}).

            \item If it is $(u',v_r+2)$ where $(u,v_r+1) = \recd(u',v_r+2)$, then $(u,v_r+1)$ is a $\rt$. Thus, the pipe in $(u,v_r+1)$ intersects pipe $\alpha_r$ again. Hence, in $D\leftarrow \binom{b}{k}$, the new insertion path will perform a $\cbswap$ and move to $(v,v_r+1)$ (see Figure \ref{fig:jdt-insert-2d}).

            \begin{figure}[h!]
                \centering
                \begin{minipage}[c]{.5\textwidth}
                    \centering
                    \includegraphics[scale = 0.4]{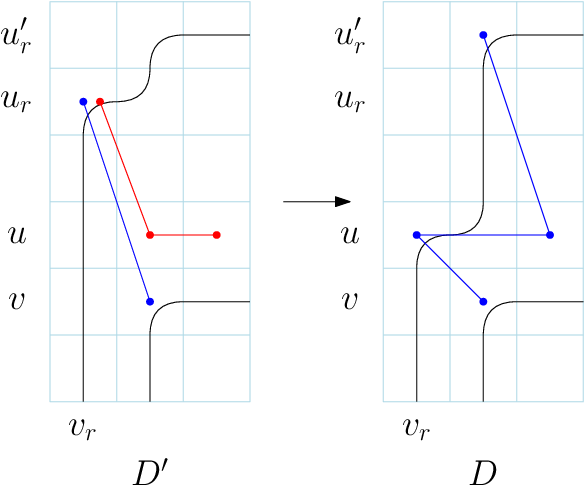}
                    \captionof{figure}{$\mathcal{J}$ moves right from $(u,v_r+1)$}
                    \label{fig:jdt-insert-2c}
                \end{minipage}%
                \begin{minipage}[c]{.5\textwidth}
                    \centering
                    \includegraphics[scale = 0.4]{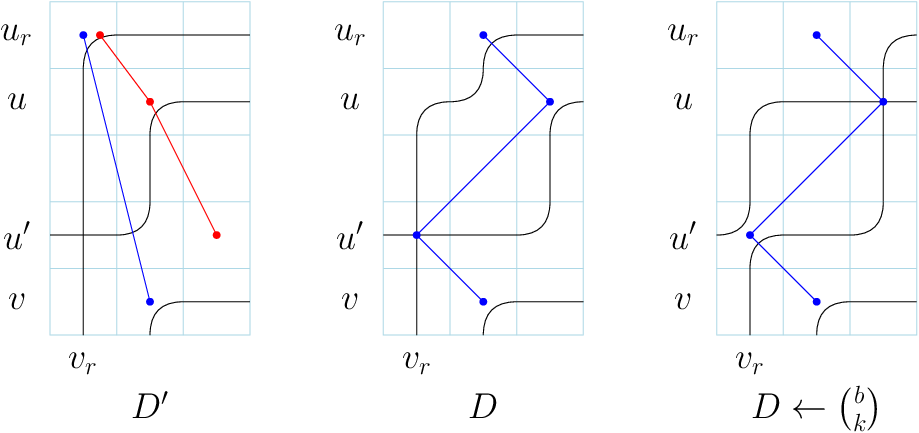}
                    \captionof{figure}{$\mathcal{J}$ $\recd$s at $(u,v_r+1)$}
                    \label{fig:jdt-insert-2d}
                \end{minipage}
            \end{figure}
        \end{itemize}

        \quad\textit{Case 2b:} $(u_r,v_r)$ is the starting point of $\mathcal{J}$. This means that pipe $v_r = \alpha_r$, and pipe $\alpha_r$ goes straight down from here. Then by Lemma \ref{lem:consec-pipe}, $\mathcal{I}$ goes through pipes $\alpha_r+1,\alpha_r+2,\ldots$. After $\mathcal{J}$ performs a $\cross$ at $(u_r,v_r)$, the new insertion path of $D\leftarrow \binom{b}{k}$ will go from $\alpha_{r-1}$ to $\alpha_r+1$ and then $\alpha_r+2,\ldots$ (see Figure \ref{fig:jdt-insert-2e}). Recall that $D\in \BPD(\pi)$ and $D'\in\BPD(\sigma)$ where $\sigma = s_i\pi$. In this case, we have $i = v_r = \alpha_r$. Recall from Case 1 that this only happens if $v_r = p_{\ell-1}$ or $v_r = p_\ell$. If $v_r = p_{\ell-1}$, then the new insertion path goes through pipes $u_1,\ldots,p_{\ell-2},p_\ell,p_\ell + 1,p_\ell+2,\ldots$ until it terminates. If $v_r = p_\ell$, then $(u_r,v_r)$ is also the endpoint of $\mathcal{I}$. Thus, the new insertion path of $D\leftarrow \binom{b}{k}$ goes from $p_{\ell-1}$ to $p_\ell + 1$ and terminates. This is consistent with the cases in Case 1.
        
        \begin{figure}[h!]
            \centering
            \includegraphics[scale = 0.4]{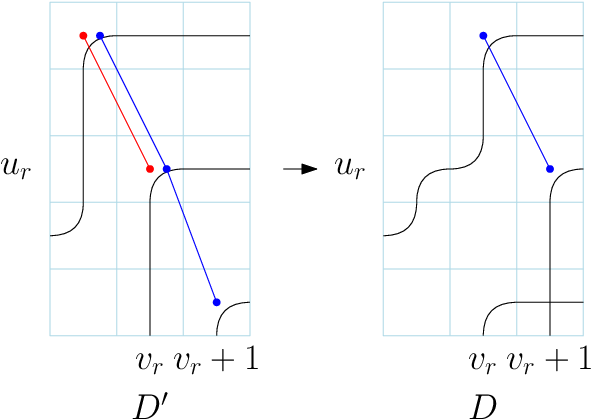}
            \caption{$\mathcal{J}$ starts from $(u_r,v_r)$}
            \label{fig:jdt-insert-2e}
        \end{figure}
        This completes the proof.
        
    \end{proof}

    \begin{lemma}
        With the same notation as in Theorem \ref{thm:local-rule}, suppose $D_\pi = \rect(D_\mu)$, then rules (1c) and (2) of Theorem \ref{thm:local-rule} is true.
    \end{lemma}

    \begin{proof}
    Recall that we consider a square 
            \[
        \begin{tikzcd}[sep=tiny]
        	\pi && \sigma \\
        	& {\ } \\
        	\mu && \rho
        	\arrow[no head, from=1-1, to=1-3]
        	\arrow[no head, from=1-1, to=3-1]
        	\arrow[no head, from=3-1, to=3-3]
        	\arrow[no head, from=1-3, to=3-3]
        \end{tikzcd}
        \]
        First of all, suppose $\mu$ is on row $a$ of the growth diagram.
    
        \textbf{Claim 1:} Rule (1c) is true. Recall that this rule is as follows. Assume there is no $\times$ in the square. If $\pi \neq \sigma, \mu$, then  $\mu = s_{i_j}\ldots s_{i_1}\pi$ where $I = \{i_j>\ldots>i_1\}$, and $\sigma = t_{\alpha\beta}\pi $ such that $\pi^{-1}(\alpha)\le k <\pi^{-1}(\beta)$ for some $\alpha<\beta$. 
        Let $x := \min(I^C \cap [\alpha,\beta))$, and $A := (I^C\cap [\beta,\infty)) \cup \{x\} = \{j_1<j_2<\ldots\}$. Then $\rho = t_{j_\ell,j_{\ell+1}}\mu$ where $\ell$ is the smallest index such that $\mu^{-1}(j_\ell) \leq k < \mu^{-1}(j_{\ell+1})$.

        Since $\sigma = t_{\alpha\beta}\pi$ for $\pi^{-1}(\alpha) \leq k < \pi^{-1}(\beta)$ for some $k$, we have $D_\sigma = D_\pi \leftarrow \binom{b}{k}$ for some $b > a$. This means that the insertion path of $D_\pi \leftarrow \binom{b}{k}$ goes through pipes $p_1<p_2<\ldots<\alpha<\beta$. Let $c$ be the number of $\bl$-tiles on row $a$ in $D_\mu$, then $D_\pi = \nabla^c(D_\mu)$, by the assumption that $D_\pi = \rect(D_\mu)$  Consider the sequence of $\BPD$s
        \[ \nabla^c(D_\mu) \xrightarrow{\Delta_{a,s_{i_c}}} \nabla^{c-1}(D_\mu) \xrightarrow{\Delta_{a,s_{i_{c-1}}}} \ldots \xrightarrow{\Delta_{a,s_{i_2}}} \nabla(D_\mu) \xrightarrow{\Delta_{a,s_{i_1}}} D_\mu. \]
        Note that this means $I = I(D_\mu) = \{i_1,\ldots,i_c\}$. Let $u,v$ be the index of the largest element in $I$ that is smaller than $\alpha,\beta$ respectively. Then, $i_c,\ldots,i_u < \alpha$. Thus, by Lemma \ref{lem:jdt-insert},  after
        \[ \nabla^c(D_\mu) \xrightarrow{\Delta_{a,s_{i_c}}} \nabla^{c-1}(D_\mu) \xrightarrow{\Delta_{a,s_{i_{c-1}}}} \ldots \xrightarrow{\Delta_{a,s_{i_{u+1}}}} \nabla^{u}(D_\mu) \xrightarrow{\Delta_{a,s_{i_u}}} \nabla^{u-1}(D_\mu), \]
        the last two pipes of $\left( \nabla^{u-1}(D_\mu) \leftarrow \binom{b}{k} \right)$ are still $\alpha$ and $\beta$. Let $x = \min(I^C\cap [\alpha,\beta))$ and $y = \min(I^C\cap [\beta,\infty))$, we have a few cases. 
        
        \textbf{Case 1:} $x = \alpha$, then $i_{u-1} > \alpha$.

        \quad\textit{Case 1a:} $y = \beta$, then none of $i_{u-1},\ldots,i_1$ is $\alpha$ or $\beta$. Hence, by Lemma \ref{lem:jdt-insert}, in $D_\mu$, the last two pipes of $\left( D_\mu \leftarrow \binom{b}{k} \right)$ are still $\alpha$ and $\beta$. 

        \quad\textit{Case 1b:} $y > \beta$, then after
        \[ \nabla^{u-1}(D_\mu) \xrightarrow{\Delta_{a,s_{i_{u-1}}}}  \ldots \xrightarrow{\Delta_{a,s_{i_v}}} \nabla^{v-1}(D_\mu), \]
        the last two pipes of $\left( \nabla^{v-1}(D_\mu) \leftarrow \binom{b}{k} \right)$ are still $\alpha$ and $\beta$ since none of $i_{u-1},\ldots,i_v$ is $\alpha$ or $\beta$. Note that $\beta = i_{v-1}$, and $\{\beta,\beta+1,\ldots,y-1\} \subseteq I$. Then in the segment
        \[ \nabla^{v-1}(D_\mu) \xrightarrow{\Delta_{a,s_{\beta}}} \nabla^{v-2}(D_\mu) \xrightarrow{\Delta_{a,s_{\beta+1}}} \ldots \xrightarrow{\Delta_{a,s_{y-2}}} \nabla^{v-y+\beta}(D_\mu) \xrightarrow{\Delta_{a,s_{y-1}}} \nabla^{v-y+\beta-1}(D_\mu), \]
        after $\Delta_{a,s_{\beta+r}}$, the last two pipes of $\left( \nabla^{v-2-r}(D_\mu) \leftarrow \binom{b}{k} \right)$ are $\alpha$ and $\beta+r+1$, by case (3) in Lemma \ref{lem:jdt-insert}. Thus, in $\nabla^{v-y+\beta-1}(D_\mu)$, the last two pipes of $\left( \nabla^{v-y+\beta-1}(D_\mu) \leftarrow \binom{b}{k} \right)$ are $\alpha$ and $y$. None of the subsequent reverse jdts performs a $\cross$ at pipe $\alpha$ of $y$, so in $D_\mu$, the last two pipes of $\left( D_\mu \leftarrow \binom{b}{k} \right)$ are still $\alpha$ and $y$.

        \textbf{Case 2:} $ x > \alpha$. This means that $\{\alpha,\alpha+1,\ldots,x-1\}\subseteq I$. Then in the segment
        \[ \nabla^{u-1}(D_\mu) \xrightarrow{\Delta_{a,s_{\alpha}}} \nabla^{u-2}(D_\mu) \xrightarrow{\Delta_{a,s_{\alpha+1}}} \ldots \xrightarrow{\Delta_{a,s_{x-2}}} \nabla^{u-x+\alpha}(D_\mu) \xrightarrow{\Delta_{a,s_{x-1}}} \nabla^{u-x+\alpha-1}(D_\mu), \]
        after $\Delta_{a,s_{\alpha+r}}$, the last two pipes of $\left( \nabla^{u-2-r}(D_\mu) \leftarrow \binom{b}{k} \right)$ are $\alpha+r+1$ and $\beta$, by case (1) in Lemma \ref{lem:jdt-insert}. Thus, in $\nabla^{u-x+\alpha-1}(D_\mu)$, the last two pipes of $\left( \nabla^{u-x+\alpha-1}(D_\mu) \leftarrow \binom{b}{k} \right)$ are $x$ and $\beta$. Repeating the same argument as in case 1 for the remaining steps, we have the last two pipes of $\left( D_\mu \leftarrow \binom{b}{k} \right)$ are $x$ and $y$.

        \textbf{Case 3:} $x$ does not exists. This means that $\{\alpha,\alpha+1,\ldots,\beta-1\}\subseteq I$. Similar to case 2, let $D_1$ be the BPD after $\Delta_{a,s_{\beta-2}}$, then the last two pipes of $\left( D_1\leftarrow \binom{b}{k} \right)$ are $\beta-1$ and $\beta$. By rule (2) in Lemma \ref{lem:jdt-insert}, let $D_2$ be the BPD after $\Delta_{a,s_{\beta-1}}$, then the insertion path of $\left( D_2\leftarrow \binom{b}{k} \right)$ goes through pipes $\beta,\beta+1,\beta+2,\ldots$ until it terminates. Repeating the same argument for $D_2$, we obtain rule (1c).

        \textbf{Claim 2:} Rule (2) is true.

        This rule is actually the same as the previous rule. Note that in this case, $b = a$, so $\pi$ is on row $b+1$. Hence, pipe $b$ is untouched in $D_\pi$. Thus, $D_\pi \leftarrow \binom{b}{k}$ starts at square $(b,b)$, which is the $\rt$-tile of pipe $b$ in column $b$. By Lemma \ref{lem:consec-pipe}, this means that $D_\pi\leftarrow \binom{b}{k}$ goes through pipes $b,b+1,b+2,\ldots$ until it terminates. By the same argument as in case 3 above, we obtain rule (2).
    \end{proof}

    The last thing we need to prove is Theorem \ref{thm:jdt}. It suffices to prove the following lemma.

    \begin{lemma}\label{lem:jdt-commute}
        Let $D\in\BPD(\pi)$ with $\pop(D) = (i,c)$, and $D' = \nabla(D)$. Given $b\geq c$ and $k$ such that the smallest descent in $\pi$ is at least $k$. Then
        \[ \nabla\left( D\leftarrow \binom{b}{k} \right) = D' \leftarrow \binom{b}{k}. \] 
    \end{lemma}

    \begin{proof}
        Suppose $\pop(D) = (i,c)$, let $\mathcal{I}$ be the insertion path $D' \leftarrow \binom{b}{k}$ and $\mathcal{J}$ be the reversed $\jdt$ path $\Delta_{c,i}(D')$. Our case analysis will be similar to the proof of Lemma \ref{lem:jdt-insert}.

        If $\mathcal{I}$ and $\mathcal{J}$ do not share any common points, then the only case we need to consider is case (1e) in Lemma \ref{lem:jdt-insert}. In this case, $\mathcal{J}$ begins with a $\cross(x,i)$ that crosses pipes $i$ and $i+1$. $\mathcal{I}$ also ends by replacing a $\bt$ between pipes $i$ and $i+1$ by a $\+$. In $D$, the insertion path of $D \leftarrow \binom{b}{k}$ performs a $\cbswap$ at this $\bt$ and moves to $\crss(x,i)$. From here, this insertion path performs some consecutive $\maxdroop$s until it terminates. Thus, $\nabla\left( D\leftarrow \binom{b}{k} \right)$ undoes these additional $\maxdroop$s, giving $D' \leftarrow \binom{b}{k}$ (see Figure \ref{fig:jdt-commute-1e}).

        \begin{figure}[h!]
            \centering
            \includegraphics[scale = 0.4]{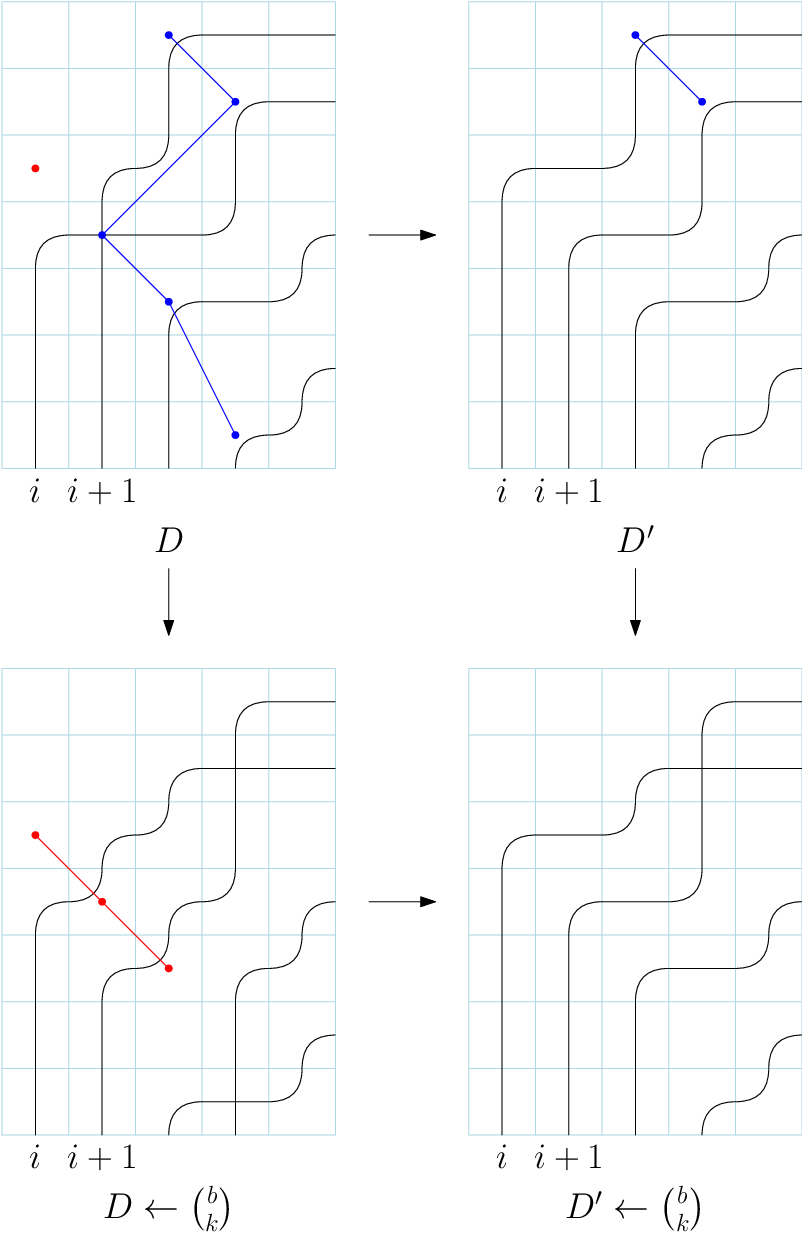}
            \caption{$\mathcal{J}$ $\crss$es at $(x,i)$}
            \label{fig:jdt-commute-1e}
        \end{figure}

        If $\mathcal{I}$ and $\mathcal{J}$ share some common points, then by Lemma \ref{lem:jdt-insert}, these points have to form a consecutive sequence of $\maxdroop$s in $\mathcal{I}$. Once again, let these points be $(u_1,v_1),\ldots,(u_r,v_r)$, and recall that for all $1\leq j < r$, $(u_{j+1},v_{j+1}) = \maxd(u_j,v_j)$, $(u_j,v_j) = \recd(u_{j+1},v_{j+1})$, and $(u_j,v_j)$ is an $\rt$-tile of some pipe $\alpha_j$. This means that pipes $\alpha_1,\ldots,\alpha_{r-1}$ are the same in $D' \leftarrow \binom{b}{k}$ and $D$. In $D$, recall that the new insertion path of $D \leftarrow \binom{b}{k}$ goes through the $\rt$-tiles of pipes $\alpha_1,\ldots,\alpha_{r-1}$ in columns $v_1+1,\ldots,v_{r-1}+1$. Specifically, it performs a $\maxdroop$ at $(u_1, v_1+1)$, which makes this square a $\bl$ in $D \leftarrow \binom{b}{k}$. Thus, the $\jdt$ path of $\nabla\left( D \leftarrow \binom{b}{k} \right)$ goes from $(u_1,v_1)$ to $(u_1,v_1+1)$ and then undoes the subsequent $\maxdroop$s (see Figure \ref{fig:jdt-commute-2a}). Note that the $\jdt$ path of $\nabla\left( D \leftarrow \binom{b}{k} \right)$ goes through exactly the points on the insertion path of $D \leftarrow \binom{b}{k}$.

        \begin{figure}[h!]
            \centering
            \includegraphics[scale = 0.4]{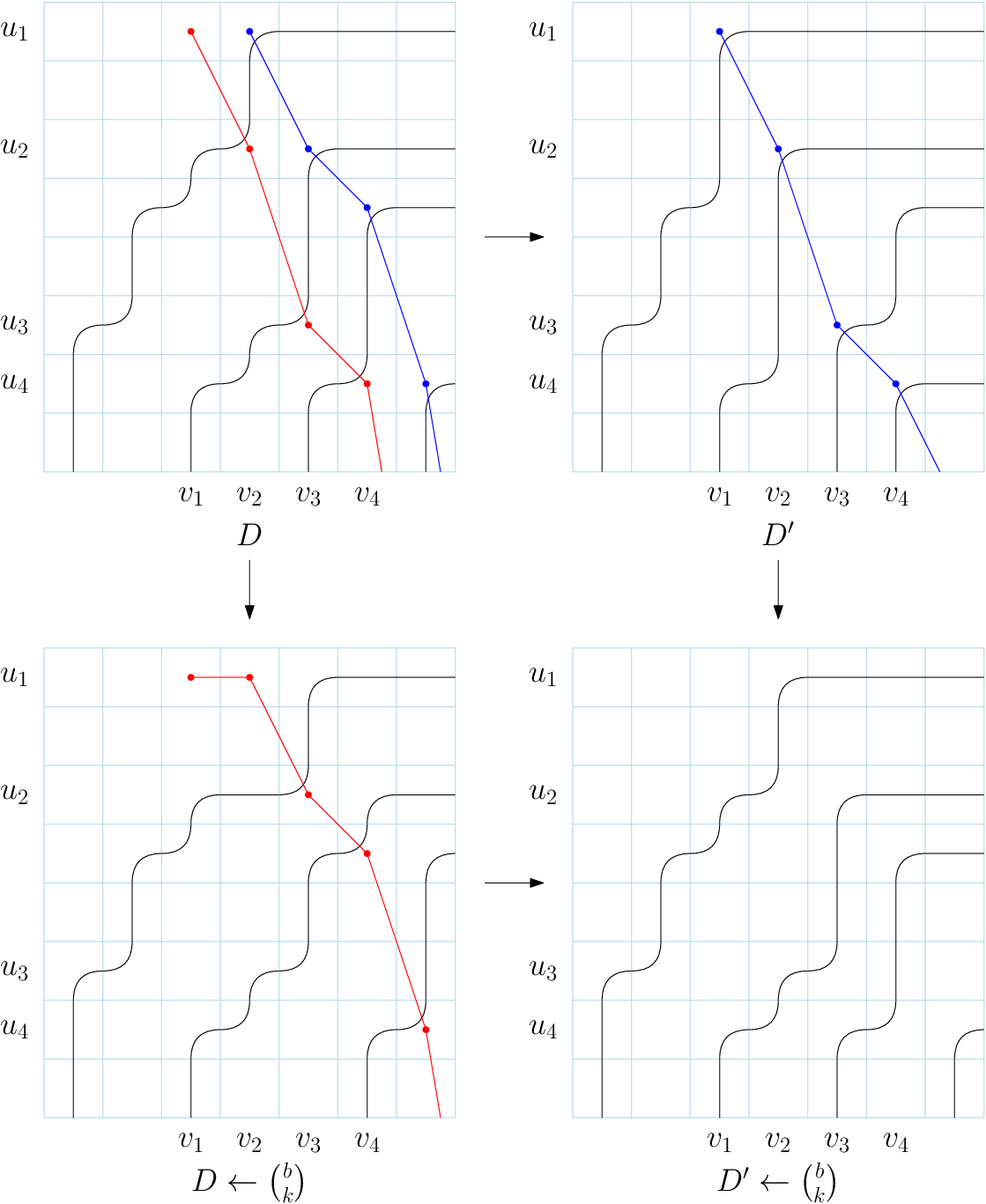}
            \caption{When $\mathcal{I}$ and $\mathcal{J}$ share common points}
            \label{fig:jdt-commute-2a}
        \end{figure}

        Finally, at $(u_r,v_r)$, we have a few cases as before. If $(u_r,v_r)$ is not the starting point of $\mathcal{J}$, then the previous point $\mathcal{J}$ is $(u,v_r+1)$ such that $(u_r,v_r) = \recd(u,v_r+1)$. We proved in Lemma \ref{lem:jdt-insert} that $(u,v_r+1)$ is also not the starting point of $\mathcal{J}$, so we consider the point prior to $(u,v_r+1)$ on $\mathcal{J}$.

        \begin{itemize}
            \item If it is $(u,v_r+2)$, then recall that the new insertion path of $D \leftarrow \binom{b}{k}$ droops from $(u_r',v_r+1)$, which is the $\rt$-tile of pipe $\alpha_r$ in column $v_r+1$, down to $(u,v_r+2)$. Thus, the $\jdt$ path of $\nabla\left( D \leftarrow \binom{b}{k} \right)$ also goes from $(u_r',v_r+1)$ to $(u,v_r+2)$ and then continue in the same path as of $\nabla(D)$ (see Figure \ref{fig:jdt-commute-2b}).

            \item If it is $(u',v_r+2)$ where $(u,v_r+1) = \recd(u',v_r+2)$, then the insertion path of $D \leftarrow \binom{b}{k}$ performs a $\maxdroop$ from $(u'_r,v_r+1)$ to $(u,v_r+2)$ followed by a $\cbswap$. Thus, in $D \leftarrow \binom{b}{k}$, pipe $\alpha_r$ turns at $(u'_r, v_r+2)$, then runs down and turns again at $(u',v_r+2)$. Thus, the $\jdt$ path of $\nabla\left( D \leftarrow \binom{b}{k} \right)$ also goes from $(u_r',v_r+1)$ to $(u,v_r+2)$ and then continue in the same path as of $\nabla(D)$ (see Figure \ref{fig:jdt-commute-2c}).
        \end{itemize}

        \begin{figure}[h!]
            \centering
            \begin{minipage}[c]{.5\textwidth}
                \centering
                \includegraphics[scale = 0.4]{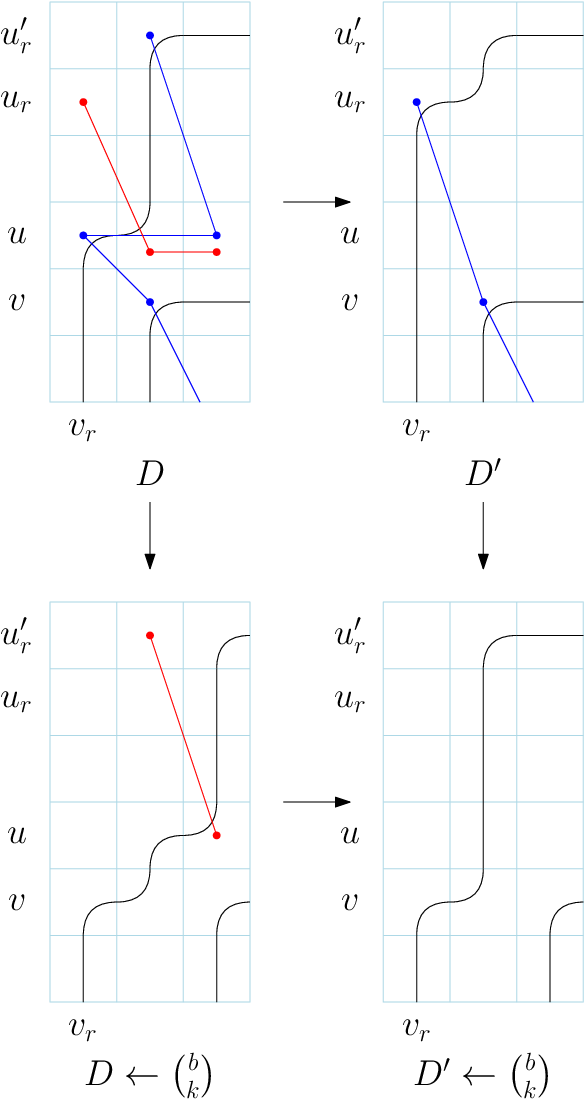}
                \captionof{figure}{$\mathcal{J}$ moves right from $(u,v_r+1)$}
                \label{fig:jdt-commute-2b}
            \end{minipage}%
            \begin{minipage}[c]{.5\textwidth}
                \centering
                \includegraphics[scale = 0.4]{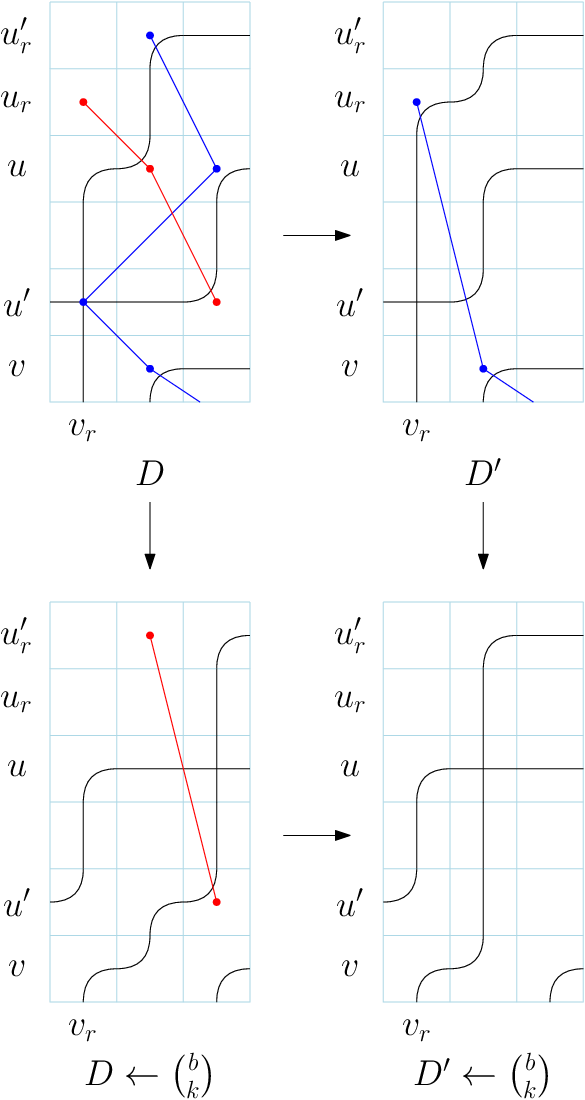}
                \captionof{figure}{$\mathcal{J}$ $\recd$s at $(u,v_r+1)$}
                \label{fig:jdt-commute-2c}
            \end{minipage}
        \end{figure}

        If $(u_r,v_r)$ is the starting point of $\mathcal{J}$, then recall that $v_r = \alpha_r$, and either $\mathcal{I}$ also ends at $(u_r,v_r)$, or it has exactly one more $\maxdroop$.

        \begin{itemize}
            \item In the former case, $D' \leftarrow \binom{b}{k}$ has pipes $\alpha_r$ and $\alpha_{r-1}$ crossing at $(u_r,v_r)$. In $D$, pipes $\alpha_r$ and $\alpha_r+1$ cross in column $v_r + 1$. The insertion path of $D \leftarrow \binom{b}{k}$ droops from $(u'_{r-1},v_r)$ down to $(u'_r,v_r+1)$ and terminates (by replacing the $\bt$ with a $\+$). Thus, in $D \leftarrow \binom{b}{k}$, the $\rt$-tile in square $(u'_{r-1},v_r+1)$ belongs to pipe $\alpha_r+1$, which goes straight down. Hence, the $\jdt$ path of $\nabla\left( D \leftarrow \binom{b}{k} \right)$ goes to square $(u'_{r-1},v_r)$ and performs an $\uncross$ here. This gives $D' \leftarrow \binom{b}{k}$ (see Figure \ref{fig:jdt-commute-2d}).

            \item In the latter case, recall that the insertion path of $D \leftarrow \binom{b}{k}$ goes from pipe $\alpha_{r-1}$ to pipes $\alpha_{r}+1,\alpha_r + 2,\ldots$ by consecutive $\maxdroop$s until it terminates. Thus, $\nabla\left( D\leftarrow \binom{b}{k} \right)$ goes to $(u'_{r-1},v_r)$ and then undoes these additional $\maxdroop$s, giving $D' \leftarrow \binom{b}{k}$ (see Figure \ref{fig:jdt-commute-2e}).
        \end{itemize}

        \begin{figure}[h!]
            \centering
            \begin{minipage}[c]{.5\textwidth}
                \centering
                \includegraphics[scale = 0.4]{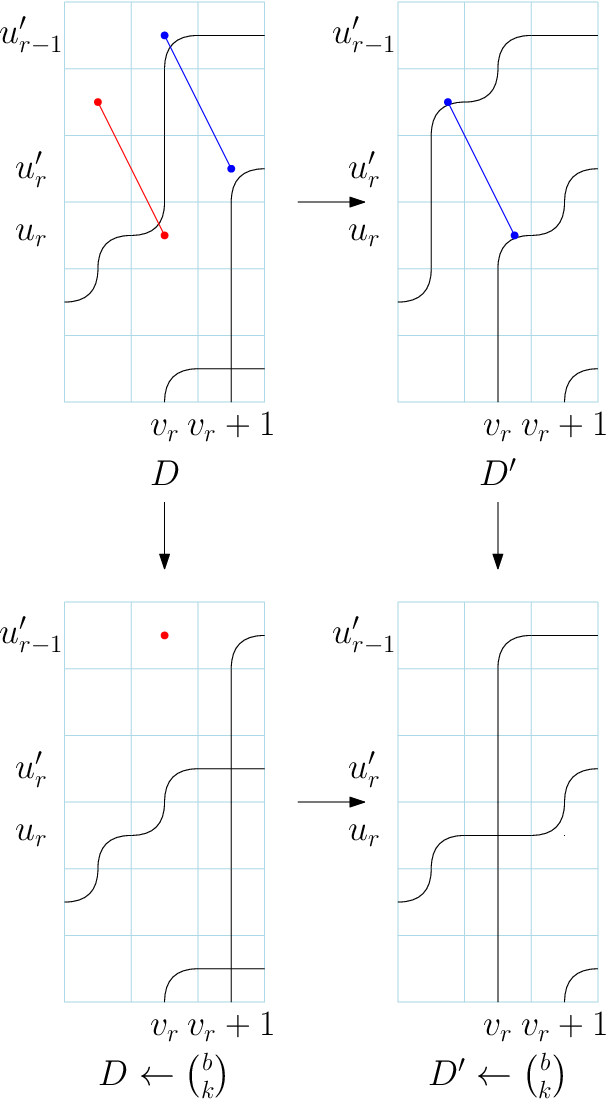}
                \captionof{figure}{$\mathcal{I}$ ends at $(u_r,v_r)$}
                \label{fig:jdt-commute-2d}
            \end{minipage}%
            \begin{minipage}[c]{.5\textwidth}
                \centering
                \includegraphics[scale = 0.4]{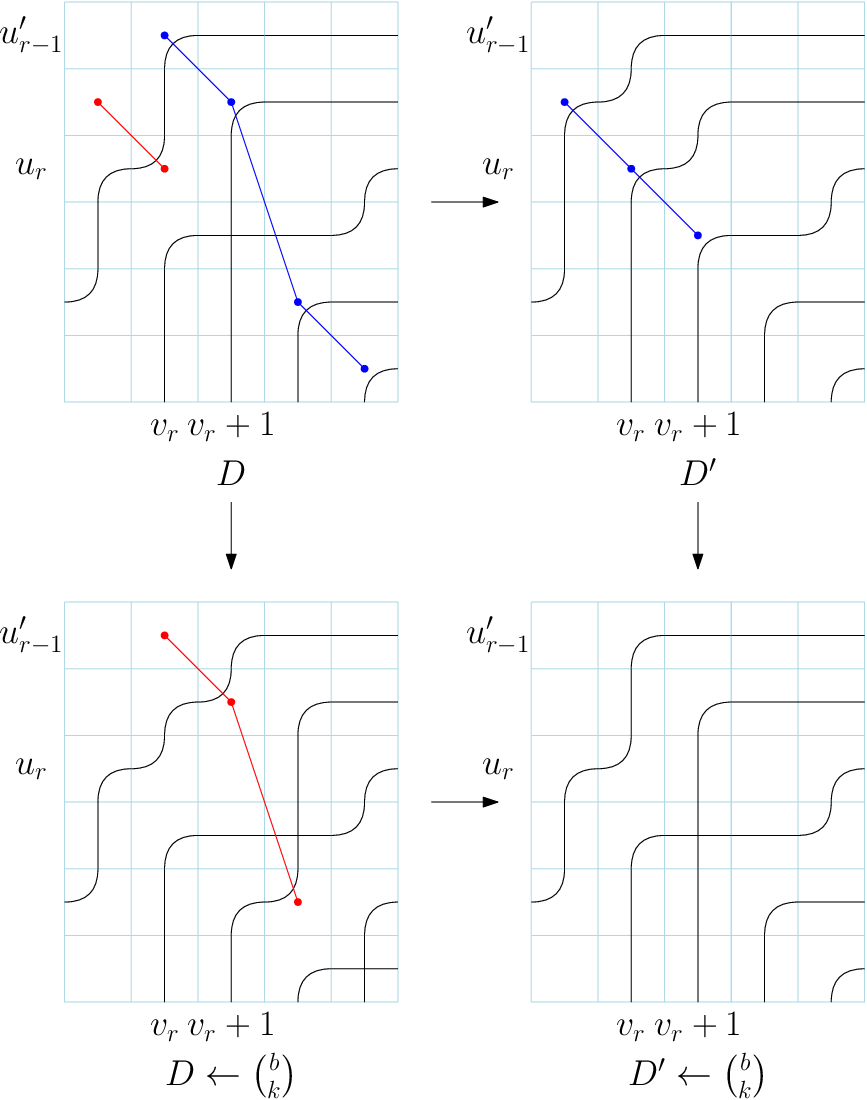}
                \captionof{figure}{$\mathcal{I}$ has exactly one more $\maxdroop$}
                \label{fig:jdt-commute-2e}
            \end{minipage}
        \end{figure}

        This completes the proof.
    \end{proof}

\bibliographystyle{alpha}
\bibliography{biblio}
\end{document}